\documentclass{amsart}[11pt,fullpage]

\usepackage{graphicx}
\usepackage{hyperref}
\usepackage{tikz}

\usepackage{mycommands}

\numberwithin{equation}{section}

\newcommand{\aug}{\text{Aug}}
\newcommand{\augrep}{KCH representation }
\newcommand{\augreps}{KCH representations }

\begin{document}

\title[KCH representations, augmentations, and $A$-polynomials]{KCH representations, augmentations, and $A$-polynomials}
\author[Christopher R. Cornwell]{Christopher R. Cornwell}

\begin{abstract}
We describe a correspondence between augmentations of knot contact homology and certain representations of the knot group. The correspondence makes the 2-variable augmentation polynomial into a generalization of the classical $A$-polynomial. It also associates to an augmentation a rank, which is bounded by the bridge number and shares its behavior under connect sums. We also study augmentations with rank equal to the braid index.
\end{abstract}

\maketitle

\section{Introduction}

Let $K$ be a knot in $\rls^3$ and denote by $\pi_K$ the fundamental group of the complement $\overline{\rls^3\setminus n(K)}$. Define an element of $\pi_K$ to be a \emph{meridian of} $K$ if it may be represented by the boundary of a disk $D$ that is embedded in $\rls^3$ and intersects $K$ at one point in the interior of $D$. Fix a field $\bb F$.

\begin{defn} If $V$ is an $\bb F$-vector space, a homomorphism $\rho:\pi_K\to\text{GL}(V)$ is a \emph{\augrep of $\pi_K$} if for a meridian $m$ of $K$, $\rho(m)$ is diagonalizable and has an eigenvalue of 1 with multiplicity $\dim V-1$. We call $\rho$ a \emph{KCH irrep} if it is irreducible as a representation.
\label{DefKCHrep}
\end{defn}

Describing the \emph{knot contact homology} $HC_*(K)$ is somewhat involved and we delay its definition until Section \ref{SecBG}. However, we remark here that $HC_*(K)$ is a non-commutative graded algebra over $\ints[U^{\pm1},\lambda^{\pm1},\mu^{\pm1}]$, and is defined as the homology of a certain differential graded algebra $(\cl A_K,\bd_K)$. An \emph{augmentation} is a graded algebra map $\epsilon:\cl A_K\to\bb F$ such that $\epsilon\circ\bd=0$ (and $\epsilon(1)=1$), where $\bb F$ has grading 0. This paper deals with the specialization of $HC_*(K)$ to an $\ints[\lambda^{\pm1},\mu^{\pm1}]$-algebra where we set $U=1$. In this case $HC_0(K)$ is isomorphic to the cord algebra $\cl C_K$ introduced in \cite{Ng08}, and we may view augmentations as homomorphisms $\epsilon:\cl C_K\to\bb F$.

It is discussed in \cite{NgSurv12} how to associate an augmentation to a KCH representation, giving a correspondence
\begin{equation}
\text{
\begin{tabular}{c c c}
$\setn{\rho:\pi_K\to\text{GL}(V)}{\rho\text{ is a KCH irrep}}$ &$\to$ & $\setn{\epsilon:\cl C_K\to\bb F}{\epsilon(\mu)\ne1}$\\
$\rho$ &$\mapsto$ &$\epsilon_\rho$
\end{tabular}
}
\label{EqnCorr}
\end{equation}

If $\rho,\rho'$ are conjugate then $\epsilon_\rho = \epsilon_{\rho'}$ (see Remark \ref{RemConjugateReps}). Our primary result shows this correspondence to be surjective.

\begin{thm} Let $\epsilon:\cl C_K\to\bb F$ be an augmentation such that $\epsilon(\mu)\ne1$. Then a KCH irrep $\rho:\pi_K\to\text{GL}(V)$ can be constructed explicitly from $\epsilon$ with the property that $\epsilon_\rho = \epsilon$. Moreover, any KCH irrep that induces $\epsilon$ is isomorphic to $(V,\rho)$.
\label{ThmMain}
\end{thm}

\begin{rem}Augmentations are geometrically motivated from the fact that Legendrian DGA's are functorial under (exact Lagrangian) cobordisms, and a cobordism to a Legendrian from the empty set induces an augmentation. See Section \ref{SecP_KInducedKCH} for how the correspondence $\rho \mapsto \epsilon_\rho$ may be viewed from this perspective.
\end{rem}

\begin{rem}We state Theorem \ref{ThmMain} for a field $\bb F$ to discuss irreducible representations and for the relationship to the $A$-polynomial. However, if we drop irreducibility and choose a ring with unity $S$, any augmentation $\epsilon:\cl C_K\to S$ (which sends $1-\mu$ to a unit) is induced from a representation $\rho:\pi_K\to\text{Aut}_S(V)$ (see Theorem \ref{ThmGenMain} and Corollary \ref{Inducing KCH reps}).
\label{RemGenKCHReps}
\end{rem}

Every \augrep $\rho$ has an eigenvalue $\mu_0\ne1$ of $\rho(m)$. An eigenvector corresponding to $\mu_0$ is also an eigenvector for $\rho(\ell)$, where $\ell$ is the preferred (Seifert-framed) longitude. Let $\lambda_0$ be the corresponding eigenvalue of $\rho(\ell)$. 

Let $\bb F=\bb C$ and write $\bb C^*$ for $\bb C\setminus\{0\}$ and define the following sets in $(\bb C^*)^2$:
\[U_K = \setn{(\lambda_0,\mu_0)}{\rho:\pi_K\to\text{GL}(V)\text{ is a KCH irrep}};\]
\[V_K = \setn{(\epsilon(\lambda),\epsilon(\mu))}{\epsilon:\cl C_K\to\bb C\text{ is an augmentation}}\setminus(\bb C^*\times\{1\}).\]

It is conjectured that, for any $K$, the maximum dimensional part of the Zariski closure of $U_K$ (resp. that of $V_K$) is a complex curve. If so a polynomial in $\ints[\lambda,\mu]$ exists with zero locus the closure of $U_K$ (resp. $V_K$). This polynomial is unique (up to a sign) once repeated factors and extraneous powers of $\lambda,\mu$ are removed and coefficients are made to be coprime. 

The polynomial for $V_K$, if multiplied by $1-\mu$, is called the \emph{2-variable augmentation polynomial} $\aug_K(\lambda,\mu)$. However, as this paper only considers augmentations sending $U$ to 1, we will not encounter its 3-variable analogue and so we refer to $\aug_K(\lambda,\mu)$ simply as the \emph{augmentation polynomial} (see \cite{AENV} for an interesting conjecture that relates the 3-variable polynomial to HOMFLY-PT polynomials).

The polynomial for $U_K$, studied in \cite{C12}, is called the \emph{stable $A$-polynomial} $\wt{A}_K(\lambda,\mu)$. The terminology ``stable'' is motivated by the bound of Theorem \ref{ThmDimBound} below. The reason for ``$A$-polynomial'' is explained as follows. A KCH representation $\rho$ may be modified to an SL$(V)$ representation by multiplying $\rho$ by some 1-dimensional representation determined by $m\mapsto (\mu_0)^{-1/d}$, $d=\dim V$. The 2-dimensional KCH representations then determine the original $A$-polynomial.

While no \emph{a priori} restriction is placed on $\dim V$ in the definition of $U_K$, it was shown in \cite{C12} that $K$ itself imposes a restriction.

\begin{thm}[\cite{C12}] Let $\set{g_1,\ldots,g_r}$ be a set of meridians that generate $\pi_K$. If $\rho:\pi_K\to\text{GL}(V)$ is a KCH irrep of $\pi_K$ then $\dim V\le r$.
\label{ThmDimBound}
\end{thm}

The reason for the bound above is that each meridian has a distinguished 1-dimensional eigenspace and the sum of these eigenspaces is an invariant subspace. This will play an important role in the proof of Theorem \ref{ThmMain} (see Section \ref{SecMerBasis}).

Theorem \ref{ThmMain} implies that $U_K=V_K$, giving us the following corollary.

\begin{cor} Given $K\subset\rls^3$ a knot, $\aug_K(\lambda,\mu)=(1-\mu)\wt{A}_K(\lambda,\mu)$ holds up to a sign.
\label{CorAugStabA}
\end{cor}

The set $U_K$ was computed in \cite{C12} for torus knots. We may now view this as a computation of the augmentation polynomial of torus knots.

\begin{cor}
Given $0<p<q$ relatively prime, let $T(p,q)$ denote the $(p,q)${--}torus knot. Then	\[\aug_{T(p,q)}(\lambda,\mu)=(1-\mu)(\lambda\mu^{(p-1)q}+(-1)^p)\prod_{n=1}^{p-1}(\lambda^n\mu^{(n-1)pq}-1).\]
\label{CorTorusKnots}
\end{cor}

In the proof of Theorem \ref{ThmMain} we construct a matrix from $\epsilon$ with rank equal to $\dim V$, where the KCH irrep corresponding to $\epsilon$ has image in $\text{GL}(V)$.

\begin{defn} The \emph{rank} of an augmentation $\epsilon:\cl C_K\to\bb F$, with the property $\epsilon(\mu)\ne 1$, is the dimension of any KCH irrep that induces $\epsilon$. The \emph{augmentation rank}, $\text{ar}(K,\bb F)$, of a knot $K$ is the maximal rank of an augmentation to $\bb F$.
\label{DefnAugRank}
\end{defn}

Let $\text{mr}(K)$ denote the \emph{meridional rank} of $K$, i.e. the minimal size of a generating set of meridians for $\pi_K$. It is well-known that $\text{mr}(K)$ is at most the bridge number $b(K)$. Recalling Theorem \ref{ThmDimBound} we have,
	\begin{equation}\text{ar}(K,\bb F) \le \text{mr}(K) \le b(K).
	\label{EqnIneqs}
	\end{equation}
Problem 1.11 in \cite{Kir95}, a question of Cappell and Shaneson that remains open, asks whether $\text{mr}(K) = b(K)$. The following result implies that, similar to bridge number, $\text{ar}(K,\bb F)-1$ is additive under connect sums.

\begin{thm}Let $K_1,K_2\subset \rls^3$ be oriented knots and suppose $\mu_0\in\bb F^*$ is such that for $n=1,2$ there is an augmentation $\epsilon_n:\cl C_{K_n}\to\bb F$ with rank $d_n$ and so that $\epsilon_n(\mu)=\mu_0\ne1$. Then $K_1\# K_2$ has an augmentation with rank $d_1+d_2-1$. Furthermore, $\text{ar}(K_1\# K_2,\bb F)=\text{ar}(K_1,\bb F)+\text{ar}(K_2,\bb F)-1$.
\label{ThmConnSum}
\end{thm}

Since $\cl C_K$ is isomorphic to $HC_0(K)\vert_{U=1}$ a study of augmentations can carried out in this setting, where the algebra is described from a closed braid representing $K$ (see Section \ref{SecHC}). This formulation allows us to obtain the following result.

\begin{thm} Suppose that $K$ is the closure of $B\in B_n$, and that $\epsilon:\cl C_K\to\bb C$ is an augmentation of $K$ with rank $n$ and $\epsilon(\mu)=\mu_0$. Then $\epsilon(\lambda)=(-\mu_0)^{-w(B)}$, where $w(B)$ is the writhe (or algebraic length) of $B$. Furthermore, there is a curve of rank $n$ augmentations in the closure of $V_K$ that corresponds to a factor $\lambda\mu^{w(B)}-(-1)^{w(B)}$ of $\aug_K(\lambda,\mu)$.
\label{ThmWrithe}
\end{thm}

\begin{cor} If $K$ is the closure of a 3-braid then
	\[\aug_K(\lambda,\mu^2) = (1-\mu^2)A_K(\lambda,\mu)B_K(\lambda,\mu^2)\]
where $A_K(\lambda,\mu)$ is the $A$-polynomial and $B_K(\lambda,\mu)$ is either 1 or $(\lambda\mu^{w(B)}\pm 1)$.
\label{CorAug3-braids}
\end{cor}
\begin{proof}Follows from Theorem \ref{ThmDimBound}, Corollary \ref{CorAugStabA}, and Theorem \ref{ThmWrithe}.\end{proof}

The hypothesis of Theorem \ref{ThmWrithe} can only possibly hold if $K$ has a braid representative with (necessarily minimal) index equal to the bridge number of $K$. In this setting the number $w(B)$ is, in fact, an invariant of $K$ by independent work in \cite{DP12} and \cite{LM12}, where the Jones Conjecture is proved.

The proof of Theorem \ref{ThmWrithe} also supplies us techniques to find knots for which the left-hand inequality in (\ref{EqnIneqs}) is strict (see Theorem \ref{Thm3bridge3braids}).

\begin{cor} If $K$ is one of the knots $\set{8_{16},8_{17},10_{91},10_{94}}$ then $2=\text{ar}(K,\bb C)<\text{mr}(K)=3$.
\label{CorARvMR}
\end{cor}

It would be very interesting if there were a coherent (or even geometric!) way to understand the absence of augmentations with rank $\text{mr}(K)$. Theorem \ref{ThmWrithe} does apply to the following family of knots. Many 3-braid closures that admit a positive or negative flype fit into this family \cite{KL99}.

\begin{thm} If $\abs u, \abs w\ge 2$, $\abs v\ge 3$, and $\delta=\pm1$ and a knot $K$ is the closure of $b=\sg_1^w\sg_2^{\delta}\sg_1^u\sg_2^v$, then the closure of $V_K$ contains a curve of rank 3 augmentations.
\label{ThmFlypes}
\end{thm}

The paper is organized as follows. In Section \ref{SecBG} we review the background on knot contact homology in our setting, particularly the cord algebra and KCH representations. Section \ref{SecKCHRepsandAugs} is dedicated to determining KCH representations from augmentations and the proof of Theorem \ref{ThmMain}. In Section \ref{SecAugReconstruct} we discuss how to build an augmentation from basic data and how augmentation rank behaves under connect sum, proving Theorem \ref{ThmConnSum}. Finally, Section \ref{SecHRkAugs} studies augmentations (particularly those of highest possible rank) from the view of a braid closure. In this section Theorem \ref{ThmWrithe} and Theorem \ref{ThmFlypes} are proved.

\subsection*{Acknowledgements}
The author would like to thank Lenhard Ng for many enlightening discussions that have been invaluable to this work. He also warmly thanks the organizers of the 2012 CAST Summer School and Conference held at the R\'enyi Institute in Budapest. This work was partly supported by NSF grant DMS-0846346 and an AMS-Simons Travel Grant.

\section{Background}
\label{SecBG}

We begin by reviewing the definition of the cord algebra $\cl C_K$ introduced in \cite{Ng08}. We also discuss two alternate constructions of the cord algebra. 

The first construction (Section \ref{SecHC}) is the (degree zero) framed knot contact homology defined in \cite{Ng08}, which will be needed in Section \ref{SecHRkAugs}. This version is the degree zero homology of the ($U=1$) combinatorial knot DGA, which is a computation of the Legendrian DGA of the conormal lift of $K$ to the unit cotangent bundle \cite{EENS12}. The conventions we use in the definition here match those from \cite{NgSurv12}. We use the notation $HC_0(K)\vert_{U=1}$ to highlight when we work with the combinatorial knot DGA construction.

A second construction we review (Section \ref{SecP_KInducedKCH}) works with the set of elements in $\pi_K$ and is our starting point for the correspondence in (\ref{EqnCorr}). In the current section and those that follow, we write $\cl P_K$ to refer to this incarnation of the cord algebra.

\subsection{The cord algebra}
\label{SubsecC_K}

Let $R_0$ be the Laurent polynomial ring $\ints[\lambda^{\pm1},\mu^{\pm1}]$. Given a knot $K\subset S^3$ with a basepoint $\ast$ on $K$, define a \emph{cord} of $(K,\ast)$ to be a path $\gamma:[0,1]\to S^3$ such that $\gamma^{-1}(K)=\set{0,1}$ and $\ast\not\in\gamma([0,1])$.

\begin{defn} Consider the noncommutative unital algebra over $R_0$ freely generated by homotopy classes of cords of $(K,\ast)$ for some choice of basepoint. (Here, homotopy of cords allows endpoints to move, but not past the basepoint.) The \emph{cord algebra} $\cl C_K$ is the quotient of this algebra by the ideal generated by relations

\begin{tikzpicture}[scale=0.5,>=stealth]
	\draw(-3,0) node {(1)}; 
	\draw[very thick,->](1,1)--(-1,-1);
	\draw[gray,->](0,0) ..controls (-1.5,0)and(0,1.5).. (0,0);
	\draw(2,-.25) node {$=$};
	\draw(3.5,-.25) node {$1-\mu$;};

	\draw(-3,-3) node {(2)}; 
	\draw[very thick,->](1,-2)--(-1,-4);
	\draw[gray,->](0.35,0.35-3) -- (-1,-2);
	\filldraw(0,-3) circle (3pt);
	\draw[very thick](.4,-3.4) node {{\huge$\ast$}};
	
	\draw(2,-.25-3) node {$=$};

	\draw(3,-.25-3) node {$\lambda$};
	\draw[very thick,->](5.5,-2)--(3.5,-4);
	\draw[gray,->](4.15,-0.35-3) -- (3.5,-2);
	\filldraw(4.5,-3) circle (3pt);
	\draw[very thick](4.9,-3.4) node {{\huge$\ast$}};

	\draw (6.5,-.25-3) node {and};
	\draw[very thick,->](9+1,-2)--(9-1,-4);
	\draw[gray,<-](9-0.35,-0.35-3) -- (9-1,-2);
	\filldraw(9,-3) circle (3pt);
	\draw[very thick](9.4,-3.4) node {{\huge$\ast$}};
	
	\draw(9+2,-.25-3) node {$=$};

	\draw(9+3,-.25-3) node {$\lambda$};
	\draw[very thick,->](9+5.5,-2)--(9+3.5,-4);
	\draw[gray,<-](9+4.85,0.35-3) -- (9+3.5,-2);
	\filldraw(9+4.5,-3) circle (3pt);
	\draw[very thick](9+4.9,-3.4) node {{\huge$\ast$}};
	
	\draw(-3,-6) node {(3)}; 
	\draw[very thick,->](1,-5)--(-1,-7);
	\draw[->, white,line width=3pt](-1,-5)--(1,-7);
	\draw[->, gray](-1,-5)--(1,-7);	
	\draw(2,-.25-6) node {$-$};
	\draw(3,-.25-6) node {$\mu$};
	\draw[->,gray](3.5,-5)--(5.5,-7);
	\draw[->,very thick,draw=white,double=black,double distance=1pt](5.5,-5)--(3.5,-7);
	\draw(6.5,-.25-6)node{$=$};
	\draw[very thick,->](9,-5)--(7,-7);
	\draw[gray,->](7,-5)--(8,-6);
	\filldraw(9.5,-6) circle (1pt);
	\draw[very thick,->](12,-5)--(10,-7);
	\draw[gray,->](11,-6)--(12,-7);
	
	\end{tikzpicture}
\label{DefnC_K}
\end{defn}

In the above definition, relations are between any cords that differ only locally, as shown. The knot $K$ is depicted more thickly. Also the figures are understood to be 3-dimensional, rather than depicting relations on planar diagrams.

\subsection{Framed knot contact homology}
\label{SecHC}

We review the construction of $HC_0(K)\vert_{U=1}$ from the combinatorial knot DGA viewpoint. We content ourselves with only defining the algebra that arises as the degree zero part of the knot DGA. For more details see \cite{NgSurv12}. 

Let $\cl A_n$ be the noncommutative unital algebra over $\ints$ freely generated by $n(n-1)$ elements $a_{ij}$, $1\le i\ne j\le n$. Let $B_n$ denote the braid group on $n$ strands. If $\sg_k$ is one of the standard generators (twisting strands in a right-handed manner), then define $\phi:B_n\to\text{Aut }\cl A_n$ by defining it on each generator as
	\[\phi_{\sg_k}:\begin{cases}
				a_{ij}\mapsto a_{ij},					& i,j\ne k,k+1\\
				a_{k+1,i}\mapsto a_{ki},				&  i\ne k,k+1\\
				a_{i,k+1}\mapsto a_{ik},				& i\ne k,k+1\\
				a_{k,k+1}\mapsto -a_{k+1,k},			&\ \\
				a_{k+1,k}\mapsto -a_{k,k+1},			&\ \\
				a_{ki}\mapsto a_{k+1,i}-a_{k+1,k}a_{ki}	&i\ne k,k+1\\
				a_{ik}\mapsto a_{i,k+1}-a_{ik}a_{k,k+1}	&i\ne k,k+1
				\end{cases}\]

Include $\iota:B_n\hookrightarrow B_{n+1}$ so that the $(n+1)^{st}$ strand does not interact, and for $B\in B_n$ let $\phi_B^*=\phi_{\iota(B)}\in\text{Aut }\cl A_{n+1}$. Define matrices $\Phi_B^L, \Phi_B^R\in\text{Mat}_{n\times n}(\cl A_n)$ by
\[\phi_B^*(a_{i,n+1})=\sum_{j=1}^n(\Phi_B^L)_{ij}a_{j,n+1},\]
\[\phi_B^*(a_{n+1,i})=\sum_{j=1}^na_{n+1,j}(\Phi_B^R)_{ji}.\]

Define an involution $x\mapsto\overline x$ on $\cl A_n$ as follows: first $\overline{a_{ij}}=a_{ji}$; then, for any $x,y\in\cl A_n$, $\overline{xy}=\overline y\hspace*{1pt}\overline x$ and extend the operation linearly to $\cl A_n$.
\begin{prop}[\cite{Ng05}, Prop.\hspace*{-0.7pt} 6.2]For a matrix of elements in $\cl A_n$, let $\overline{M}$ be the matrix such that $\left(\overline M\right)_{ij} = \overline{M_{ij}}$. Then for $B\in B_n$, $\Phi_B^R$ is the transpose of $\overline{\Phi_B^L}$.
\label{Prop:Transpose}
\end{prop}

Finally, define a matrix ${\bf A}$ by 
\begin{equation}{\bf A}_{ij}=\begin{cases}a_{ij}	& i<j\\
						-\mu a_{ij}	& i>j\\
						1-\mu		& i=j\end{cases},
\label{EqnAmatrixDefn}
\end{equation}
and, given $B\in B_n$, the diagonal matrix ${\bf\Lambda}=\text{diag}[\lambda\mu^w,1,\ldots,1]$ where $w$ is the writhe (algebraic length) of $B$. Extend the map $\phi_B$ to $\cl A_n\otimes R_0$ so that it fixes $\lambda,\mu$.

\begin{defn}Let $K=\widehat{B}$ be the (braid) closure of $B\in B_n$ and let $\cl I_B$ be the ideal in $\cl A_n \otimes R_0$ generated by entries in the matrices ${\bf A}-{\bf\Lambda}\cdot\phi_B({\bf A})\cdot{\bf\Lambda}^{-1}, {\bf A}-{\bf\Lambda}\cdot\Phi^L_B\cdot{\bf A}$, and ${\bf A}-{\bf A}\cdot\Phi^R_B\cdot{\bf\Lambda}^{-1}$. The algebra $(\cl A_n\otimes R_0)/\cl I_B$ is the degree zero homology of the combinatorial knot DGA, denoted $HC_0(K)\vert_{U=1}$. 
\label{DefHC_0}
\end{defn}
In Definition \ref{DefHC_0} (and throughout the paper), given a homomorphism $f$ defined on the entries of a matrix $M$ we use $f(M)$ for the matrix obtained by applying $f$ to the entries. The following was proved in \cite{Ng08} (see also \cite[\S 3, 4]{NgSurv12}).

\begin{thm} There is an isomorphism $F_{HC}:\cl C_K\to HC_0(K)\vert_{U=1}$ of $R_0$-algebras.
\label{ThmHC0}
\end{thm}

For the discussion in Section \ref{SecHRkAugs} we need to define $F_{HC}$ on generators of $\cl C_K$. We view the braid $B$ as horizontal with strands oriented to the right and numbered to be increasing from top to bottom. Consider a flat disk $D$, to the right of the braid, with $n$ punctures where it intersects $K=\widehat{B}$ (see Figure \ref{FigA_nGens}). We assume that the $n$ punctures of $D$ are collinear, on a line that separates $D$ into upper and lower half-disks. Denote by $c_{ij}$, a cord of $(K,\ast)$ that is contained in the upper half-disk of $D$, and has initial endpoint on the $i^{th}$ strand and terminal endpoint on the $j^{th}$ strand. The cord algebra $\cl C_K$ is generated by the homotopy classes of the set $\{c_{ij}, 1\le i\ne j\le n\}$ and $F_{HC}$ is defined by $F_{HC}(c_{ij}) = {\bf A}_{ij}$ (as in (\ref{EqnAmatrixDefn})).

\begin{figure}[ht]
\begin{tikzpicture}[scale=0.5,>=stealth]
	\clip (-.7,-1.25) rectangle (7.3,2.5);
	\draw(0,-.25)--(0,2.25)--(3.5,2.25)--(3.5,-.25)--cycle;
	\draw(1.75,1) node {{\Large $B$}};
	\foreach \y in {0,.5,1,1.5,2}
		\filldraw (4.75,\y) circle (2pt);
	\foreach \y in {0,.5,1,1.5,2}
		\draw(-.35,\y) -- (0,\y);
	\foreach \y in {0,.5,1,1.5,2}	
		\draw[thin] (3.5,\y) -- (4.75,\y);
	\foreach \y in {0,.5,1,1.5,2}	
		\draw[->] (5.18,\y) -- (5.7,\y);
	\foreach \y in {0,.5,1,1.5,2}
		\draw (5.7,\y) ..controls (7.7+\y*1.3,\y) and (7.7+\y*1.3,-\y-1) ..(3,-\y-1)
			 		  ..controls (-3-\y*1.3,-\y-1) and (-2.7-\y*1.3,\y)..(-.7,\y);
	\draw[red,->]
			(4.75,.5) ..controls (5.05,.75) and (5.05,1.25) ..(4.75,1.5);
	\draw[draw=white,double=black,very thick]
			(4.45,2.1)  ..controls (4.5,2.35) and (4.6,2.5) .. (4.8,2.5)
					  ..controls (5,2.5) and (5.2,2) .. (5.2 ,1)
					  ..controls (5.2,0) and (5,-.5) .. (4.8,-.5)
					  ..controls (4.6,-.5) and (4.5,-.35)..(4.45,-.1);
	\draw 	(4,-.6) node {{\footnotesize $D$}};
	\draw	(6.15,1.95) node {{\large $\ast$}};
	\draw	(-.65,.5) node {{\footnotesize $i$}}
			(-.65,1.5) node{{\footnotesize $j$}};
\end{tikzpicture}
\caption{Cord $c_{ij}$ of $K=\widehat B$}
\label{FigA_nGens}
\end{figure}
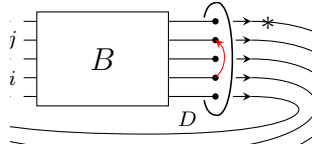

To understand $\phi_B$ from this perspective, view $c_{ij}$ as a path in $D$. Considering $B$ as a mapping class of the punctured disk $D$, let $B\cdot c_{ij}$ denote the isotopy class (fixing endpoints) of the path to which $c_{ij}$ is sent. Viewing $D$ from the left (the side from which the strands of $B$ point towards $D$), $\sg_k$ acts by rotating the $k\textrm{-}$ and $(k+1)\textrm{-}$punctures an angle of $\pi$ about their midpoint in counter-clockwise fashion. 

Following \cite[Section 2]{Ng05b}, consider the set $P(D)$ of isotopy classes of embedded (oriented) paths in $D$ with endpoints on distinct punctures. There is a unique map $\psi:P(D)\to \cl A_n$ which satisfies $\psi(c_{ij})=a_{ij}$ if $i<j$, $\psi(c_{ij})=-a_{ij}$ if $i>j$, and such that $\psi(B\cdot c_{ij}) = \phi_B(\psi(c_{ij}))$ for any $B\in B_n$. In addition, given representative paths of elements of $P(D)$ which differ only near a puncture as depicted in Figure \ref{FigRelnPathAlg}, the relation in Figure \ref{FigRelnPathAlg} is satisfied by the corresponding images under $\psi$. In Section \ref{SecHRkAugs} we use this characterization of $\phi_B$ to justify some calculations of the matrix $\phi_B({\bf A})$.

\begin{figure}[ht]
\begin{tikzpicture}[scale=0.5,>=stealth]
	\filldraw	(-1,1) circle (3pt)
				(3.65,1) circle (3pt)
				(7.75,1) circle (3pt)
				(9.65,1)circle(3pt);
	\filldraw 	(8.525,1) circle(1pt);
	\draw[thick,->]	
				(-2,1) ..controls (-1.5,1) and (-1.5,.5) ..(-1,.5)
					   ..controls (-.5,.5) and (-.5,1) ..(0,1);
	\draw[thick,->]	
				(2.65,1) ..controls (3.15,1) and (3.15,1.5) ..(3.65,1.5)
					   ..controls (4.15,1.5) and (4.15,1) ..(4.65,1);
	\draw[thick,->]	(6.9,1) -- (7.8,1);
	\draw[thick,->]	(9.65,1) -- (10.65,1);
	\draw	(1.25,1) node {$=$}
			(5.65,1) node {$-$};
	\foreach \x in {-2,2.65,6.95,9.65}
	\draw	(\x,1) node[left] {$\psi\big[$};
	\foreach \x in {0,4.65,7.85,10.4}
	\draw 	(\x,1) node[right] {$\big]$};			
\end{tikzpicture}
\caption{Relation in the image of $\psi$}
\label{FigRelnPathAlg}
\end{figure}
To see that $F_{HC}$ produces an isomorphism one must use a framed version of $\psi$, incorporating $\mu$ into the map (see \cite[\S 3.2]{Ng08} for details).

We also require the following results, originally proved in \cite{Ng05}. Following the terminology from that paper, we refer to Theorem \ref{ThmChainRule} as the Chain Rule.

\begin{thm} Let $B,B'$ be braids in $B_n$. Then $\Phi_{BB'}^L = \phi_B(\Phi_{B'}^L)\cdot\Phi_B^L$ and $\Phi_{BB'}^R = \Phi_B^R\cdot\phi_B(\Phi_{B'}^R)$.
\label{ThmChainRule}
\end{thm}

\begin{cor} The matrices $\phi_B(\Phi_{B^{-1}}^L)$ and $\phi_B(\Phi_{B^{-1}}^R)$  are the inverse of $\Phi_B^L$ and $\Phi_B^R$, respectively.
\label{CorPhiInv}
\end{cor}

\begin{thm} Let ${\bf A}$ be the matrix defined in (\ref{EqnAmatrixDefn}). Then for any $B\in B_n$,
	\[\phi_B({\bf A}) = \Phi_B^L\cdot {\bf A}\cdot \Phi_B^R.\]
\label{ThmPhiSqRt}
\end{thm}

Theorem \ref{ThmPhiSqRt} implies that the ideal $\cl I_B$ used to define $HC_0(K)\vert_{U=1}$ is generated by entries in ${\bf A}-{\bf\Lambda}\cdot\Phi^L_B\cdot{\bf A}$ and ${\bf A}-{\bf A}\cdot\Phi^R_B\cdot{\bf\Lambda}^{-1}$ only.

\subsection{Augmentations from KCH representations}
\label{SecP_KInducedKCH}
We now review the algebra $\cl P_K$, which is defined with elements of $\pi_K$, and describe the correspondence in (\ref{EqnCorr}). 

\begin{defn} Let $P_K$ denote the underlying set of the knot group $\pi_K$, where we write $[\gamma]\in P_K$ for $\gamma\in\pi_K$. In $\pi_K$ let $e$ be the identity, $m$ a choice of meridian, and $\ell$ the preferred longitude of $K$. Define $\cl P_K$ to be the noncommutative unital algebra freely generated over $R_0$ by $P_K$ modulo the relations:
	\en{
	\item $[e]=1-\mu$;
	\item $[m\gamma]=\mu[\gamma], [\gamma m]=[\gamma]\mu$ and $[\ell\gamma]=\lambda[\gamma], [\gamma\ell]=[\gamma]\lambda$, for any $\gamma\in\pi_K$;
	\item $[\gamma_1\gamma_2]-[\gamma_1m\gamma_2]=[\gamma_1][\gamma_2]$ for any $\gamma_1,\gamma_2\in\pi_K$.
	}
\label{DefnPi1Cords}
\end{defn}

\begin{thm}[\cite{Ng08}] $\cl C_K$ and $\cl P_K$ are isomorphic as $R_0$-algebras.
\label{ThmPi1Cords}
\end{thm}

The isomorphism of the theorem $F_{\cl P}:\cl P_K\to\cl C_K$ may be defined as follows. Suppose the basepoint $x$ for the group $\pi_K$ is on the boundary torus. Choose a fixed path $p$ from a point on $K$ to $x$ with interior in the tubular neighborhood $n(K)$. Let $\overline p$ denote $p$ with reversed orientation. If $g\in\pi_K$ is represented by a loop $\gamma$, define $F_{\cl P}([g])$ to be $\mu^{\text{lk}(\gamma,K)}$ times the cord given by the concatenation $p\gamma\overline p$ (here $\text{lk}(\gamma,K)$ is the linking number of $\gamma$ and $K$). We use this definition in Section \ref{SecHRkAugs}.

This identification of $\cl C_K$ with $\cl P_K$ uses a basepoint for $\pi_K$, as would be expected by the choice of $m$ in the definition of $\cl P_K$. The oriented boundary of a meridian disk of $n(K)$ that contains $x$ in its boundary is a representative of $m$.

\begin{rem}That $\cl P_K$ is defined to be an algebra over $R_0$ along with (2) implies the relations $[m\gamma]=[\gamma m]$ and $[\ell\gamma]=[\gamma\ell]$ for any $\gamma\in\pi_K$. 

Alternatively, construct the unital algebra $\wt{\cl P}_K$ freely generated over $\ints$ by $P_K\cup\{\lambda^{\pm1},\mu^{\pm1}\}$, modulo the relations (1), (2), (3), and the relation $\lambda\mu=\mu\lambda$. Then $\wt{\cl P}_K$ is isomorphic to the degree zero homology of the \emph{fully noncommutative knot DGA} (see the discussion in \cite{NgSurv12}).

Working with $\wt{\cl P}_K$, an analogue of Theorem \ref{ThmPi1Cords} has been found in work of K.\ Cieliebak, T.\ Ekholm, J.\ Latschev, and L.\ Ng.
\label{RemFullyNonComm}
\end{rem}

\begin{thm}[\cite{CELN}] There is an injective ring homomorphism $\wt{\cl P}_K\hookrightarrow \ints[\pi_K]$ with image generated by the peripheral subgroup $\langle\ell,m\rangle\subset \pi_K$ and elements of the form $\gamma-m\gamma$ where $\gamma\in\pi_K$.
\label{ThmGpRing}
\end{thm}

\subsection{Geometric view of augmentations}
Here we provide a rough description of a geometric source of augmentations. For more details see \cite[Section 6]{AENV}. 

Given a Legendrian $\Lambda$ in a contact manifold $Y$, and under some conditions on the pair $(Y,\Lambda)$ (see e.g.\ \cite{EES07}), there is an associated Legendrian (or Chekanov-Eliashberg) DGA $(\mathcal A(\Lambda), \bd(\Lambda))$ which, with an appropriate notion of equivalence, is invariant under Legendrian isotopy \cite{EES07}. 

As alluded to at the start of Section \ref{SecBG} the definition of $HC_0(K)\vert_{U=1}$ comes from a computation of $(\cl A(\Lambda_K),\bd(\Lambda_K))$, where $\Lambda_K$ is the unit conormal lift of $K$, which is a Legendrian in the standard contact structure on the unit cotangent bundle $ST^*\rls^3$. We remark that this does fall into the setting of \cite{EES07}, since $ST^*\rls^3$ is contactomorphic to the 1-jet bundle $J^1(S^2)$. In the following, we wish to work with the \emph{fully noncommutative} version of the Legendrian DGA, denoted $\wt{\cl A}(\Lambda_K)$ (cf.\ \cite[Remark 2.2]{EENS12}).

The DGA construction produces a contravariant functor from the category of Legendrians and exact Lagrangian cobordisms to the category of differential graded algebras. In particular, an exact Lagrangian \emph{filling} $L$ {--} a cobordism from the empty set {--} of a Legendrian $\Lambda$ induces a DGA map from $\wt{\cl A}(\Lambda)$ to the ground ring, which is identified with the DGA of the empty set, with zero differential. This induced map is a chain map, hence it is an augmentation.

In the symplectization of $ST^*\rls^3$, $\Lambda_K$ admits an exact Lagrangian filling $M_K$ with the topology of the knot complement. While the augmentation induced from $M_K$ has little information, one can keep track of the homotopy class in $\pi_1(M_K)=\pi_K$ of the boundary of rigid holomorphic disks and obtain a homomorphism $\Phi:\wt{\cl A}(\Lambda_K) \to \ints[\pi_K]$ such that $\Phi\circ\bd_K=0$. Consideration of 1-parameter families of holomorphic disks shows the image of $\Phi$ is the subring of $\ints[\pi_K]$ indicated in Theorem \ref{ThmGpRing}. It is conjectured that $\Phi$ induces an isomorphism on zero-graded homology, which would give a symplecto-geometric source for Theorem \ref{ThmGpRing} (this is not the approach taken by Cieliebak, Ekholm, Latschev, and Ng).

Let $\rho:\pi_K\to\text{GL}(V)$ be a KCH representation, $\mu_0$ the eigenvalue of $\rho(m)$ not equal to 1. The longitude $\ell$ commutes with $m$, hence the $\mu_0$-eigenspace of $\rho(m)$ is preserved by $\rho(\ell)$. Extending $\rho$ to $\ints[\pi_K]$, the definition of a KCH representation implies that $\rho(m), \rho(\ell),$ and $\rho(\gamma - m\gamma)$ (for any $\gamma\in \pi_K$) are each a linear map preserving the 1-dimensional $\mu_0$-eigenspace, and so each restricted to that eigenspace corresponds to multiplication by an element of $\bb F$. This lets us assign a scalar to each element in the image of $\Phi$. Identifying that image with $\wt{\cl P}_K$ (via Theorem \ref{ThmGpRing}), this assignment agrees with the definition of $\epsilon_\rho$ presented in Proposition \ref{ThmKCHReps}. Any augmentation induced from a KCH representation thus arises from a flat connection on $M_K$. Theorem \ref{ThmMain} says that all augmentations with $\epsilon(\mu)\ne1$ arise in this way.

\subsection{The augmentation induced from a KCH representation}
Let $S$ be a ring with 1. Generalize Definition \ref{DefKCHrep} by letting $V$ be a right $S$-module. In this context, we say $\rho:\pi_K\to\text{Aut}_S(V)$ is a KCH representation if there is $\mu_0$ in $S$ such that $1-\mu_0$ is invertible and there is a generating set $\set{e_1,\ldots, e_r}$ for $V$ such that $\rho(m)e_1=e_1\mu_0$ and $\rho(m)e_i=e_i$ for $2\le i\le r$. As in the introduction, there is a $\lambda_0$ such that $\rho(\ell)e_1=e_1\lambda_0$ since $\rho(m)$ commutes with $\rho(\ell)$; also $\mu_0,\lambda_0$ are units as $\rho(m)$ and $\rho(\ell)$ are invertible.

\begin{prop} If $\rho:\pi_K\to\text{Aut}_S(V)$ is a KCH representation with $\text{Ann}_S(e_1)=\{0\}$, then there is an induced augmentation $\epsilon_{\rho}:\wt{\cl P}_K\to S$ with $\epsilon_\rho(\mu)=\mu_0$ and $\epsilon_\rho(\lambda)=\lambda_0$.
\label{ThmKCHReps}
\end{prop}

\begin{proof}If $V$ is free then the proof of Theorem \ref{ThmKCHReps} is the same as that which is outlined for $S=\bb C$ in \cite{NgSurv12}. In this case, one chooses a basis of $V$ from the set $\{e_1,\ldots,e_r\}$ (which contains $e_1$ by force). Define a bilinear form so that $\langle e_i,e_j\rangle=\delta_{ij}$ on this basis (where $\delta_{ij}$ is the Kronecker-delta). Then $\epsilon_\rho$ is defined by setting $\epsilon_\rho(\mu)=\mu_0$, $\epsilon_\rho(\lambda)=\lambda_0$, and $\epsilon_\rho([\gamma])=(1-\mu_0)\langle\rho(\gamma)e_1,e_1\rangle$ for $\gamma\in\pi_K$. The map $\epsilon_\rho:\wt{\cl P}_K\to S$ is then determined. That $V$ is a right module is relevant to $\epsilon_\rho$ being well-defined. For example, if $\langle\rho(\gamma)e_1,e_1\rangle = s$ then $\langle\rho(m\gamma)e_1,e_1\rangle = \mu_0s$ uses the right action of $S$.

The definition of $\epsilon_\rho$ above is equally valid when $V$ is not free, despite $\langle v,w\rangle$ not being well-defined for general $w\in V$. For suppose that $\sum_{k=1}^re_kb_k = v = \sum_{k=1}^re_kc_k$ for elements $b_k,c_k\in S$, $k=1,\ldots r$. Then
	\[0 = \sum_{k=1}^re_k(b_k-c_k) - \rho(m)\sum_{k=1}^re_k(b_k-c_k) = e_1(1-\mu_0)(b_1-c_1).\]
As $\text{Ann}_S(e_1)=\{0\}$ and $1-\mu_0$ is invertible, $b_1=c_1$.\end{proof}

\begin{rem}Given another KCH representation $\rho':\pi_K\to\text{Aut}_S(V')$ and a linear isomorphism $\varphi:V'\to V$ such that $\varphi\circ\rho'(g) = \rho(g)\circ\varphi$ for all $g\in\pi_K$, the vectors $e_i' = \varphi^{-1}(e_i)$, $i=1,\ldots,r$, are eigenvectors of $\rho'(m)$ (with the same eigenvalue as $e_i$). Noting how the bilinear form we used depends on these eigenvectors, this implies that $\epsilon_\rho = \epsilon_{\rho'}$.
\label{RemConjugateReps}
\end{rem}

\begin{rem} There is a different version of Proposition \ref{ThmKCHReps} for representations of $\pi_K$ that generalize KCH representations. Suppose that there is a basis of an $\bb F$-vector space $V$ and $\rho:\pi_K\to\text{GL}(V)$ with which $\rho(m)={\footnotesize\begin{pmatrix}M_0&0\\ 0&\text{Id}\end{pmatrix}}$ for a $k\times k$ invertible matrix $M_0$ with $\text{Id}_k-M_0$ also invertible. Let $W\subset V$ be the subspace spanned by the first $k$ basis vectors. Then $\rho$ induces an augmentation $\epsilon:\wt{\cl P}_K\to\text{Mat}_k(\bb F)$ by setting $\epsilon(\mu)=M_0$ and $\epsilon([g])=(\text{Id}_W-M_0)\text{Proj}_W\rho(g)$.
\end{rem}

\section{\augreps and augmentations}
\label{SecKCHRepsandAugs}

In this section, after some inital remarks, we prove Theorem \ref{ThmMain} by constructing a representation from a certain universal augmentation. The construction is the content of Theorem \ref{ThmGenMain} in Section \ref{SecKCHFromAug}. In Section \ref{SecMerBasis} we then restrict to the case that the target of our augmentation is a field and address the irreducibility of inducing KCH representations. The proof of Theorem \ref{ThmMain} appears at the end of Section \ref{SecMerBasis}.

Fix a meridian $m$ of $K$. Consider a set $\Gamma=\{\gamma_1,\ldots,\gamma_r\}\subset\pi_K$, with $\gamma_1$ the identity, such that $\cl G=\setn{g_i}{g_i=\gamma_i^{-1}m\gamma_i, 1\le i\le r}$ generates $\pi_K$. 

Note from Theorem \ref{ThmPi1Cords} that any augmentation $\epsilon:\wt{\cl P}_K\to S$ has values that, for any $g,h\in\pi_K$, satisfy
	\begin{equation}	
	\begin{aligned}
	\epsilon([e])=1-\epsilon(\mu),\quad &\epsilon([mg])=\epsilon(\mu[g]), \epsilon([gm])=\epsilon([g]\mu),\\
					\text{and }\epsilon([gmh])&=\epsilon([gh])-\epsilon([g])\epsilon([h]).
	\end{aligned}
	\label{AugvalueRelns}
	\end{equation}

\subsection{The universal augmentation of $K$}
\label{SecUnivAug}

Define $E = \{[e]^n \mid n\ge 0\}$ and let $\cl Q_{K}$ be the localization $E^{\text{-}1}\cl P_K$. In $\cl P_K$ the set $E$ satisfies Ore's condition, hence the localization homomorphism is injective and $\cl Q_K$ is isomorphic to a ring of fractions. We may also define $\wt{\cl Q}_K$, the localization of $\wt{\cl P}_K$ with respect to $E$, though $E$ does not satisfy Ore's condition in $\wt{\cl P}_K$; for example, see \cite{Cohn95}. Though we cannot consider $\wt{\cl Q}_K$ as a ring of fractions containing $\wt{\cl P}_K$, the localization $\iota:\wt{\cl P}_K\to\wt{\cl Q}_K$ has the expected universal property, that if $f:\wt{\cl P}_K\to S$ is an $E$-inverting homomorphism then there is a unique $g:\wt{\cl Q}_K\to S$ such that $g\circ\iota = f$.

\begin{defn}The \emph{universal augmentation of $K$} is the map $\iota:\wt{\cl P}_K\to\wt{\cl Q}_{K}$, which may not be injective.
\label{DefUnivAug}
\end{defn}

The augmentations we consider all factor through the universal augmentation; further, if $\epsilon(\mu), \epsilon(\lambda)$ are central then they define a map on $\cl P_K$ and factor through $\cl P_K\hookrightarrow\cl Q_K$. We will abuse notation, writing $[g]$ for the image $\iota([g])$ in $\wt{\cl Q}_K$.

\subsection{Notation and setup}
\label{SecNotationSetup} Consider the direct sum $\oplus^r\wt{\cl Q}_K$ as a right $\wt{\cl Q}_K$-module. Write $[\Gamma g]$ for the element $([\gamma_1g],\ldots,[\gamma_rg])$ in $\oplus^r\wt{\cl Q}_K$. Also, for $1\le j\le r$, we define $v_j = [\Gamma\gamma_j^{-1}]$ and let $V$ be the right submodule over $\wt{\cl Q}_K$ generated by $\{v_1,v_2,\ldots,v_r\}$.

We need some preparatory lemmas to prove Theorem \ref{ThmGenMain}, which shows that the universal augmentation is induced from a representation. 

\begin{lem}For $h,h'\in\pi_K$, if $g=hh'$ then there are elements $c_1,\ldots,c_r\in\wt{\cl P}_K$ such that $[\Gamma g]=\sum_{k=1}^r[\Gamma h\gamma_k^{-1}]c_k$. Thus $[\Gamma g]\in V$.
\label{LemEltExp}
\end{lem}
\begin{proof} As $h'$ is a product of elements in $\cl G$, we write $h'=g_{i_1}^{\ve_1}\ldots g_{i_l}^{\ve_l}$, with $\ve_k=\pm1$ for $1\le k\le l$. In $\wt{\cl P}_K$, $[a\gamma_{i_k}^{-1}m^{\ve_k}\gamma_{i_k}b]=[ab] -\ve_k[a\gamma_{i_k}^{-1}][m^{{\tiny \frac{\ve_k-1}2}}\gamma_{i_k}b]$ for any $a,b\in\pi_K$ and $1\le k\le l$. Hence
	\[[\Gamma hh'] = [\Gamma h\gamma_1^{-1}]+\sum_{k=1}^l-\ve_k[\Gamma h\gamma_{i_k}^{-1}]([m^{{\tiny \frac{\ve_k-1}2}}\gamma_{i_k}w(k)])\]
where $w(k)=g_{i_{k+1}}^{\ve_{k+1}}\ldots g_{i_l}^{\ve_l}$. Taking $h=e$ and using $\iota$ gives that $[\Gamma g]\in V$.\end{proof}
\begin{rem}The elements $c_1,\ldots,c_r$ are chosen independently of $h$.
\label{RemEltExp}
\end{rem}

\begin{lem} Suppose that $c_1,\ldots,c_r\in\wt{\cl Q}_K$ are such that $\sum_{k=1}^rv_kc_k = 0$. If $g\in\pi_K$ then $\sum_{k=1}^r[\Gamma g\gamma_k^{-1}]c_k =0$.
\label{LemEltExp2}
\end{lem}
\begin{proof}The statement trivially holds if $g$ is the identity since $v_k=[\Gamma\gamma_k^{-1}]$. Let $g=g_i^\ve g'$ for some $g_i\in\cl G$ and $\ve=\pm1$. Letting $\delta=(\ve-1)/2$, if we suppose that the statement holds for $g'$ then
	\[\sum_{k=1}^r[\Gamma g\gamma_k^{-1}]c_k = \sum_{k=1}^r[\Gamma g'\gamma_k^{-1}]c_k -\ve [\Gamma\gamma_i^{-1}]\mu^\delta\sum_{k=1}^r[\gamma_i g'\gamma_k^{-1}]c_k = 0,\]
as $\sum[\gamma_i g'\gamma_k^{-1}]c_k$ is the $i$ coordinate of $\sum[\Gamma g'\gamma_k^{-1}]c_k=0$.
\end{proof}

\subsection{KCH representations from augmentations}
\label{SecKCHFromAug}

In this section we show that the universal augmentation is induced from a KCH representation, from which it will follow that the same is true of any augmentation sending $1-\mu$ to a unit.

\begin{thm} There is a well-defined KCH representation $\rho_\iota:\pi_K\to\text{Aut}_{\wt{\cl Q}_K}(V)$ that induces the universal augmentation $\iota:\wt{\cl P}_K\to\wt{\cl Q}_K$.
\label{ThmGenMain}
\end{thm}
\begin{rem}The analogous statement for $\cl P_K\hookrightarrow\cl Q_K$ also holds.
\end{rem}

\begin{proof}
Given $g\in\pi_K$ define $\rho_\iota(g)v_j = [\Gamma g\gamma_j^{-1}]$, for each $1\le j\le r$, which is an element of $V$ by Lemma \ref{LemEltExp}. Extend $\rho_\iota(g)$ to a $\wt{\cl Q}_K$-linear map. By Lemma \ref{LemEltExp2}, this determines a well-defined map on $V$.

We show below that $\rho_\iota(gh)=\rho_\iota(g)\rho_\iota(h)$, and it follows that $\rho_\iota(g)$ is invertible and $\rho_\iota:\pi_K\to\text{Aut}_{\wt{\cl Q}_K}(V)$ is a well-defined homomorphism.

Fix $1\le j\le r$. Choose any $h,h''\in\pi_K$ and set $h'=h''\gamma_j^{-1}$. Using the calculation from Lemma \ref{LemEltExp} we write $[\Gamma hh']$ as a sum $\sum_{k=1}^r[\Gamma h\gamma_k^{-1}]c_k$. By our definitions this implies
	\[\rho_\iota(hh'')v_j = [\Gamma hh'] = \sum_{k=1}^r[\Gamma h\gamma_k^{-1}]c_k=\rho_\iota(h)\left(\sum_{k=1}^rv_kc_k\right).\]
By Remark \ref{RemEltExp} the $c_k$ are independent of $h$. Hence we may set $h=e$ in the above equation and obtain $\sum_{k=1}^rv_kc_k=[\Gamma h']=\rho_\iota(h'')v_j$. Hence $\rho_\iota(hh'')=\rho_\iota(h)\rho_\iota(h'')$, showing $\rho_\iota$ is a homomorphism.

Recall that $\gamma_1=e$ which makes $g_1=m$. To see that $\rho_\iota$ is a \augrep we find a generating set $\{e_1,\ldots,e_r\}$ for $V$ as discussed in Section \ref{SecP_KInducedKCH}.

Set $e_1=v_1$ and $e_j=v_j-v_1(1-\mu)^{-1}[\gamma_j^{-1}]$ for $j=2,\ldots,r$. One finds that $\rho_\iota(m)e_1=e_1\mu$ and $\rho_\iota(m)e_j=e_j$ for $j=2,\ldots,r$, as $\rho_\iota(m)v_j = [\Gamma g_1\gamma_j^{-1}] = v_j - v_1[\gamma_j^{-1}]$ by (3) in Theorem \ref{ThmPi1Cords}.

To determine the induced augmentation (which by Proposition \ref{ThmKCHReps} exists since the first coordinate of $e_1$ is invertible), for given $g\in\pi_K$ we choose $c_1,\ldots,c_r\in\cl P_K$ as in Lemma \ref{LemEltExp}, with $h=e$, so that
	\al{
	\rho_\iota(g)e_1 = [\Gamma g] = \sum_{k=1}^rv_kc_k 	&= e_1\left(c_1 + (1-\mu)^{-1}\sum_{k=2}^r[\gamma_k^{-1}]c_k\right) + \sum_{k=2}^re_kc_k\\
										&= e_1(1-\mu)^{-1}\sum_{k=1}^r[\gamma_k^{-1}]c_k + \sum_{k=2}^re_kc_k,
	}
the last equality since $[\gamma_1^{-1}]=[e]=1-\mu$. Now, the first coordinate of $[\Gamma g]\in V$ is $[g]$ and the first coordinate of $\sum_{k=1}^rv_kc_k$ is $\sum_{k=1}^r[\gamma_k^{-1}]c_k$. Thus, $\rho_\iota(g)e_1 = e_1(1-\mu)^{-1}[g]+\sum_{k=2}^r e_kc_k$. Hence $(1-\mu)\langle\rho_\iota(g)e_1,e_1\rangle = [g] \in\wt{\cl Q}_K$ and the induced augmentation is $\iota:\wt{\cl P}_K\to\wt{\cl Q}_K$.
\end{proof}

\begin{cor}Let $\epsilon:\wt{\cl P}_K\to S$ (or, alternatively $\epsilon:\cl P_K\to S$) be an augmentation with $\epsilon(1-\mu)$ invertible in $S$. Then there is a KCH representation $\rho:\pi_K\to\text{Aut}_S(W)$ for some right $S$-module $W$, such that $\epsilon_\rho = \epsilon$.
\label{Inducing KCH reps}
\end{cor}
\begin{proof}There is a unique homomorphism $\epsilon':\wt{\cl Q}_K\to S$ such that $\epsilon'\circ\iota = \epsilon$. Let $V$ be as in Theorem \ref{ThmGenMain} and define $W$ to be generated over $S$ by vectors $w_j=\epsilon'(v_j)$ (here we apply $\epsilon'$ to each coordinate). The map $\rho(g)$ is defined by setting $\rho(g)w_j = \epsilon'(\rho_\iota(g)v_j)$. From Theorem \ref{ThmGenMain} it follows that $\rho$ is well-defined and $\epsilon_\rho([g]) = \epsilon([g])$.
\end{proof}
\begin{rem} Note that $w_j=\rho(\gamma_j^{-1})w_1$ for each $j=1,\ldots,r$.\label{RemInducingKCH}\end{rem}

\subsection{KCH representations on vector spaces and the meridian subspace}
\label{SecMerBasis}

We restrict our attention to the case $S=\bb F$ is a field. Let $\rho:\pi_K\to\text{GL}(V)$ be any KCH representation, $\dim V=d$. Take a basis for $V$ of eigenvectors $e_1,\ldots,e_d$ of $\rho(m)$ such that $\rho(m)e_1=\mu_0e_1$.
\begin{defn} Define $w_j = \rho(\gamma_j^{-1})e_1$ for each $1\le j\le r$. Define $W_{\rho}(\Gamma)=\bb F\langle w_1,w_2,\ldots,w_r\rangle$ to be the \emph{meridian subspace} of $V$.
\label{defn:MerSbSpace}
\end{defn}

\begin{lem}For $1\le i\le r$, the vector $w_i$ satisfies $\rho(g_i)w_i=\mu_0w_i$ and $W_{\rho}(\Gamma)$ is an invariant subspace.
\label{lem:BasisEigenvecs}
\end{lem}
\begin{proof}
For each $1\le i\le r$ we have
\[\rho(g_i)w_i = \rho(\gamma_i^{-1}m)e_1 = \mu_0w_i.\]
By definition of $w_j$, we find that $w_j-\rho(g_i)w_j=\rho(\gamma_i^{-1})(\text{Id}_V-\rho(m))\rho(\gamma_i\gamma_j^{-1})e_1$ for each $1\le j\le r$. In addition, if $w = \sum_{k=1}^dc_ke_k$ then $(\text{Id}_V-\rho(m))w = (1-\mu_0)c_1e_1$. Taking $w=\rho(\gamma_i\gamma_j^{-1})e_1$, this indicates the equality
	\begin{equation}
w_j - \rho(g_i)w_j = (1-\mu_0)\langle\rho(\gamma_i\gamma_j^{-1})e_1,e_1\rangle w_i = \epsilon_\rho([\gamma_i\gamma_j^{-1}])w_i.
	\label{MerActionEqn}
	\end{equation}

where, as in Section \ref{SecP_KInducedKCH}, $\langle\cdot,\cdot\rangle$ is the bilinear form on $V$ given by the Kronecker-delta $\langle e_i,e_j\rangle=\delta_{ij}$. This proves the lemma since $\cl G$ generates $\pi_K$.
\end{proof}

We remark that equation (\ref{MerActionEqn}) will be important in Lemma \ref{DimWAndEpsDeg} below. The following lemma was shown for $\bb F=\bb C$ in \cite[\S 3.2]{C12}. The proof given there carries over to our setting. 

\begin{lem} If $\rho:\pi_K\to\text{GL}(V)$ is a KCH representation and $W\subset V$ a subspace on which the action of $\rho$ is the identity, then the quotient representation $\overline\rho:\pi_K\to\text{GL}(V/W)$ is a KCH representation and $\epsilon_\rho = \epsilon_{\overline\rho}$.
\label{LemTrivSubspc}
\end{lem}

\begin{lem} Let $W\subseteq V$ be an invariant subspace of $\rho:\pi_K\to\text{GL}(V)$. Then either the action of $\rho$ restricted to $W$ is the identity, or $W_{\rho}(\Gamma)\subset W$.
\label{lem:MerSubspaceMinimal}
\end{lem}
\begin{proof}
Given $x\in W$ we have that $(1-\mu_0)\langle x,e_1\rangle e_1 = (\text{Id}_V-\rho(m))x \in W$. Thus, as $1-\mu_0$ is a unit, either $e_1\in W$ or $\langle x,e_1\rangle=0$ for every $x\in W$. 

If $e_1\in W$, then $w_i=\rho(\gamma_i^{-1})e_1\in W$ for $i=1,\ldots,r$. This implies $W_{\rho}(\Gamma)\subset W$.

Alternatively it must be that $\rho(m)x=x$ for every $x\in W$. But then, for any $g\in\pi_K$ and any $x\in W$, we have $\rho(g^{-1}mg)x = x$ since $\rho(g)x\in W$. As $\pi_K$ is generated by conjugates of $m$, the action of $\rho$ on $W$ is the identity.
\end{proof}

Given an augmentation $\epsilon:\cl C_K\to\bb F$ that is induced from a KCH representation $\rho:\pi_K\to\text{GL}(V)$, Lemma \ref{LemTrivSubspc} states that if the action on $W\subset V$ is the identity then the quotient representation induces the same augmentation. Taking such a quotient sufficiently many times gives a KCH representation $\rho':\pi_K\to\text{GL}(V')$ such that $V'$ has no such subspace. Restricting to the meridian subspace $W_{\rho'}(\Gamma)$, and applying Lemmas \ref{lem:MerSubspaceMinimal} and \ref{lem:BasisEigenvecs}, we obtain the following corollary.

\begin{cor} If $\epsilon:\cl C_K\to\bb F$ is induced from a KCH representation, it is induced from a KCH irrep on the meridian subspace of some KCH representation.
\label{cor:W(G)IsIrrep}
\end{cor}

Now for an augmentation $\epsilon:\cl C_K\to\bb F$, define $\cl E(\Gamma)$ to be the $r\times r$ matrix over $\bb F$ having columns $\epsilon([\Gamma\gamma_j^{-1}])$ (note that these are the vectors $w_j$ from the KCH representation defined in Section \ref{SecKCHFromAug}).

\begin{lem} Let $\rho:\pi_K\to\text{GL}(V)$ be any KCH representation that induces $\epsilon$ and suppose $V$ has no proper subspace where the action of $\rho$ is the identity. Then
	\begin{enumerate}
		\item the rank of $\cl E(\Gamma)$ equals $\dim W_{\rho}(\Gamma)$;
		\item the dimension of any KCH irrep inducing $\epsilon$ equals $\dim W_{\rho}(\Gamma)$.
	\end{enumerate} 
\label{DimWAndEpsDeg}
\end{lem}

\begin{proof}
As previously we write $w_j=\rho(\gamma_j^{-1})e_1$ for $j=1,\ldots,r$, and the vectors $w_1,\ldots,w_r$ generate $W_{\rho}(\Gamma)$. For $1\le j\le r$ let $c_j\in\bb F$ be such that $\sum_j c_jw_j=0$. From equation (\ref{MerActionEqn}) we see that for each $g_i\in\cl G$
\[0=(Id-\rho(g_i))\sum_{j=1}^r c_jw_j = \sum_{j=1}^r \epsilon([\gamma_i\gamma_j^{-1}])c_jw_i.\]
As no $w_i$ is zero, it must be that $\sum_j \epsilon([\gamma_i\gamma_j^{-1}])c_j=0$ for each $i$. This implies that $\dim W_{\rho}(\Gamma)$ is at least the rank of $\cl E(\Gamma)$.

In addition, if there are scalars $c_j, 1\le j\le r$ such that $\sum_j \epsilon([\gamma_i\gamma_j^{-1}])c_j=0$ for each $i$, then the same equality implies that $\rho(g_i)\sum_j c_jw_j = \sum_j c_jw_j$ for each $i$. Then $\sum_j c_jw_j\in W_{\rho}(\Gamma)$ is fixed by $\pi_K$ and so $\sum c_jw_j=0$ by hypothesis. Hence $\dim W_{\rho}(\Gamma)\le \text{rank of }\cl E(\Gamma)$, so they are equal.

To see (2) holds, if $\sg:\pi_K\to\text{GL}(X)$ is any KCH irrep inducing $\epsilon$ then $W_\sg(\Gamma)$, which is nonzero and invariant, is $X$. As no proper subspace of $X$ is acted upon trivially, $\dim X= \text{rank}(\cl E(\Gamma))=\dim W_{\rho}(\Gamma)$ by (1).
\end{proof}

We can now prove our main theorem.

\vspace*{6pt}
\noindent{\bf Theorem \ref{ThmMain}.}\ \
{\it Let $\epsilon:\cl C_K\to\bb F$ be an augmentation such that $\epsilon(\mu)\ne1$. Then a KCH irrep $\rho:\pi_K\to\text{GL}(V)$ can be constructed explicitly from $\epsilon$ with the property that $\epsilon_\rho = \epsilon$. Moreover, any KCH irrep that induces $\epsilon$ is isomorphic to $(V,\rho)$.
}
\vspace*{6pt}

\begin{proof}[Proof of Theorem \ref{ThmMain}]Given $\epsilon:\cl C_K\to\bb F$ with $\epsilon(\mu)\ne1$, there is a KCH representation $\rho$ that induces $\epsilon$, constructed in Corollary \ref{Inducing KCH reps}. This representation acts on the vector space generated by the columns of $\cl E(\Gamma)$, denoted by $V$. By Corollary \ref{cor:W(G)IsIrrep} a subspace of a quotient of $V$ is a KCH irrep that induces $\epsilon$, and by Lemma \ref{DimWAndEpsDeg} this KCH irrep also has dimension equal to the rank of $\cl E(\Gamma)$. Thus the original $V$ was irreducible.

If $\rho':\pi_K\to\text{GL}(V')$ is any other KCH irrep that induces $\epsilon$ then the meridian subspace is $V'$. Extend the assignment $w_j=\rho'(\gamma_j^{-1})e_1 \mapsto \epsilon([\Gamma\gamma_j^{-1}])$, for each $1\le j\le r$, to a linear map $\psi:V'\to V$. Examining the proof of Lemma \ref{DimWAndEpsDeg} we see $\psi$ is a well-defined isomorphism of vector spaces. It is also $\pi_K$-equivariant since (\ref{MerActionEqn}) and (\ref{AugvalueRelns}) together imply that $\psi\circ\rho' = \rho \circ \psi$ on a generating set of $\pi_K$.
\end{proof}

\section{Constructing augmentations from a matrix}
\label{SecAugReconstruct} 
As in the previous section, fix a meridian $m$ of $K$ and a set $\Gamma=\{\gamma_1,\ldots,\gamma_r\}\subset\pi_K$, with $\gamma_1$ the identity, such that $\cl G=\setn{g_i}{g_i=\gamma_i^{-1}m\gamma_i, 1\le i\le r}$ generates $\pi_K$. We find criteria sufficient for an $r\times r$ matrix over $\bb F$ to be $\cl E(\Gamma)$ (defined in Section \ref{SecMerBasis}) for some augmentation $\epsilon:\cl C_K\to\bb F$. The criteria are used to prove Theorem \ref{ThmConnSum}. We begin with the following lemma.

\begin{lem} Let $\epsilon,\epsilon':\cl C_K\to S$ be two augmentations satisfying $\epsilon([\gamma_i\gamma_j^{-1}])=\epsilon'([\gamma_i\gamma_j^{-1}])$ for every pair $1\le i,j\le r$, and such that $\epsilon(1-\mu)$ is invertible. Then $\epsilon'=\epsilon$.
\label{LemEGensAug}
\end{lem}
\begin{proof} Since $\cl P_K$ is generated by elements $[g]\in\pi_K$ over $R_0$ \comment{(see Theorem \ref{ThmPi1Cords})}we need to check that $\epsilon,\epsilon'$ agree on $\mu,\lambda$ and any $[g]\in\pi_K$. The equality $[\gamma_1\gamma_1^{-1}]=[e]=1-\mu$ implies $\epsilon(\mu)=\epsilon'(\mu)$ by our assumption.

Given $g\in\pi_K$, choose a product equal to $g$ of elements in $\cl G$ and their inverses. From the form of the elements $c_k$ determined in the proof of Lemma \ref{LemEltExp}, iterating the process there determines an expansion of $[g]=[\gamma_1g]$ solely in terms of powers of $\mu$ and elements $[\gamma_i\gamma_j^{-1}]$, $1\le i,j\le r$. The assumption on the homomorphisms $\epsilon$ and $\epsilon'$ implies that $\epsilon([g])=\epsilon'([g])$. 

Finally, $\epsilon(\lambda)\epsilon(1-\mu)=\epsilon([\ell])=\epsilon'([\ell])=\epsilon'(\lambda)\epsilon(1-\mu)$ and $\epsilon(1-\mu)$ is invertible. The result follows.
\end{proof}

\subsection{Matrices that determine an augmentation}
\label{SecAugReconstruct1}

Consider now each element of $\cl G$ (and the inverses) as formal words $g_i^{\pm1}=\gamma_i^{-1}m^{\pm1}\gamma_i$, $i=1,\ldots,r$. For $\cl R$, a set of words in $\cl G$ and its inverses, denote the set of formal inverses of words in $\cl R$ by $\cl R^{-1}$. Given an $r\times r$ matrix $E$ write $E_{ij}$ for the $(i,j)$-entry in $E$. Following the proof of Lemma \ref{LemEGensAug}, use the explicit expansion determined by iterating Lemma \ref{LemEltExp} and the assignments $\epsilon(\gamma_i\gamma_j^{-1})=E_{ij}$ and $\epsilon(\mu)=1-E_{11}$ to assign a value $\epsilon_E(g)$ to a word $g=\gamma_1 g\gamma_1^{-1}$, written in $\{\gamma_i\gamma_j^{-1} | 1\le i,j\le r\}\cup\{m^{\pm1}\}$. For $[g]\in\pi_K$, the element represented by the word $g$, we find conditions on $E$ such that $\epsilon([g])=\epsilon_E(g)$ determines a well-defined augmentation.

\begin{lem} Let $\langle\cl G\mid\cl R\rangle$ be a presentation of $\pi_K$ with meridian generators and notation as above. Given an $r\times r$ matrix $E$, define a function $\epsilon:\cl C_K\to\bb F$ by setting $\epsilon(\mu)=1-E_{11}$ and $\epsilon([g])=\epsilon_E(g)$ for $[g]\in\pi_K$. Then $\epsilon$ is a well-defined augmentation if the diagonal entries of $E$ are all equal and not 0 or 1, and $\epsilon_E(\gamma_iR\gamma_j^{-1})=E_{ij}$ for every $R\in\cl R\cup \cl R^{-1}$ and each $1\le i,j\le r$.
\label{LemAugFomE}
\end{lem}
\begin{proof}We have defined $\epsilon(\mu)=\mu_0$ so that $1-\mu_0$ equals any diagonal entry of $E$.

Let $F$ be the free group generated by $\cl G$ and its inverses. Since $\pi_K\cong\langle\cl G\mid\cl R\rangle$ is the quotient of $F$ by the smallest normal subgroup containing $\cl R$, if $g$ and $h$ represent the same element in $\pi_K$ there is a finite sequence of allowed moves on the words $g$ and $h$, after which the resulting words agree. The allowed moves are the insertion or deletion into a word of either (1) a cancelling pair $xx^{-1}$ or $x^{-1}x$, $x\in\cl G$, or (2) an element of $\cl R$ or its inverse in $\cl R^{-1}$.

We must show that $\epsilon([g])=\epsilon([h])$ when $[g]=[h]$ in $\pi_K$. By the previous paragraph it is sufficient to prove that if $h$ may be obtained from $g$ by just one allowable insertion move then $\epsilon_E(g)=\epsilon_E(h)$. First suppose the insertion is a cancelling pair. Let $\varepsilon=\pm1$ and $\delta_\pm=(\pm \varepsilon-1)/2$ and apply (\ref{AugvalueRelns}) and $\epsilon_E(\gamma_k\gamma_k^{-1})=E_{kk}=1-\mu_0$ to see that
	\al{\epsilon_E(\gamma_i\gamma_k^{-1})\epsilon_E(\gamma_k g_k^{-\varepsilon}\gamma_j^{-1})
			&= \epsilon_E(\gamma_i\gamma_k^{-1})\left(\epsilon_E(\gamma_k\gamma_j^{-1}) +\varepsilon\mu_0^{\delta_-}\epsilon_E(\gamma_k\gamma_k^{-1})\epsilon_E(\gamma_k\gamma_j^{-1})\right)\\
			&= \mu_0^{-\varepsilon}\epsilon_E(\gamma_i\gamma_k^{-1})\epsilon_E(\gamma_k\gamma_j^{-1}).
	}
From this, and the fact that $\delta_+-\varepsilon = \delta_-$, we calculate
\al{\epsilon_E(\gamma_i g_k^\varepsilon g_k^{-\varepsilon}\gamma_j^{-1})	
		&= \epsilon_E(\gamma_i g_k^{-\varepsilon}\gamma_j^{-1}) -\varepsilon\mu_0^{\delta_+-\varepsilon} \epsilon_E(\gamma_i\gamma_k^{-1})\epsilon_E(\gamma_k\gamma_j^{-1})\\
		&= \epsilon_E(\gamma_i\gamma_j^{-1}).
	}

Now to compare $\epsilon_E(g)$ to $\epsilon_E(h)$ we expand both words to be expressed completely in terms of $\set{\epsilon_E(\gamma_i\gamma_j^{-1})\mid 1\le i,j\le r}$ except that for some $k,i_0,j_0$ the expansion of $h$ involves $\epsilon_E(\gamma_{i_0}g_k^{\pm1}g_k^{\mp1}\gamma_{j_0}^{-1})$. But we've shown our procedure assigns this the same value as $\epsilon_E(\gamma_{i_0}\gamma_{j_0}^{-1})$, and thus $\epsilon_E(g)=\epsilon_E(h)$.
	
	If the insertion is of the second type, there is some $R\in\cl R\cup\cl R^{-1}$ inserted into $h$, and this is the only difference between $g$ and $h$ as words. Applying the same argument as in the previous paragraph, and using the assumption that $\epsilon_E(\gamma_{i_0}R\gamma_{j_0}^{-1})=E_{i_0j_0}=\epsilon_E(\gamma_{i_0}\gamma_{j_0}^{-1})$, we see that $\epsilon_E(g)=\epsilon_E(h)$. 
	
	By the definition of $\epsilon$	and the isomorphism $\cl C_K\cong\cl P_K$, we only need to check that $\epsilon([\ell g])=\epsilon([g\ell])$ for any $g\in\pi_K$. But since $\ell$ commutes with $m$,
		\[\epsilon(\mu^{-1}[\ell])\epsilon([g])=\epsilon(\mu^{-1}[\ell g])-\epsilon(\mu^{-1}[\ell mg])=(\mu_0^{-1}-1)\epsilon([\ell g]).\]
	By considering $\epsilon([g])\epsilon(\mu^{-1}[\ell])$, this also equals $(\mu_0^{-1}-1)\epsilon([g\ell])$. Thus $\epsilon([\ell g])=\epsilon([g\ell])$ since $\mu_0\ne1$.
\end{proof}

\subsection{Connect sums}
\label{SecConnectSums}
Recall the definition of the rank of an augmentation, and the augmentation rank of a knot (Definition \ref{DefnAugRank}). In \cite[Prop.\ 5.8]{Ng08} the augmentation variety $V_{K_1\# K_2}$ was related to $V_{K_1}$ and $V_{K_2}$. Here we relate the rank of the augmentations in each.

\vspace*{6pt}
\noindent{\bf Theorem \ref{ThmConnSum}.}\ \
{\it Let $K_1,K_2\subset \rls^3$ be oriented knots and suppose $1\ne\mu_0\in\bb F^*$ is such that for $n=1,2$ there is an augmentation $\epsilon_n:\cl C_{K_n}\to\bb F$ with rank $d_n$ and so that $\epsilon_n(\mu)=\mu_0$. Then $K_1\# K_2$ has an augmentation with rank $d_1+d_2-1$. Furthermore, $\text{ar}(K_1\# K_2,\bb F)=\text{ar}(K_1,\bb F)+\text{ar}(K_2,\bb F)-1$.
}
\vspace*{6pt}

\begin{proof}
For $n=1,2$ let $\langle\cl G_n\mid\cl R_n\rangle$ be a presentation of $\pi_{K_n}$, where $\cl G_n$ is a set of meridians. Define $r_n=\abs{\cl G_n}$. Let $m_n$ be a meridian such that $[m_nh]=\mu[h]$ (for all $h\in\pi_{K_n}$) in the cord algebra $\cl C_{K_n}$. We may assume that $m_n\in\cl G_n$. Define $\cl G=(\cl G_1\setminus\{m_1\})\cup(\cl G_2\setminus\{m_2\})\cup\{m\}$. Order $\cl G$ so that $m=g_1$ is first and the last $r_2-1$ elements are in $\cl G_2$. Write $g_i=\gamma_i^{-1}m\gamma_i$ for the $i^{th}$ generator, $i=1,\ldots,r_1+r_2-1$.

The group $\pi_{K_1\# K_2}$ has a presentation $\langle\cl G\mid\cl R\rangle$ where $\cl R=\cl R_1\cup\cl R_2$ except that in each $\cl R_n$, $m_n$ is replaced by $m$. We may take $\cl R$ to have $r_1+r_2-2$ relators (with $r_n-1$ from each summand). In fact, we may assume (and do) each relator to have the form $R=g_{\ell}g_mg_{\ell}^{-1}g_k^{-1}$, where $g_k, g_{\ell}, g_m\in\cl G_n$ if $R\in\cl R_n$.
 
Define an $(r_1+r_2-1)\times(r_1+r_2-1)$ matrix $E$ by setting 
	\[(E)_{ij} =\begin{cases} \epsilon_1([\gamma_i\gamma_j^{-1}])\text{, if }i,j\le r_1\\
\epsilon_2([\gamma_i\gamma_j^{-1}])\text{, if }i,j>r_1\\
\frac1{1-\mu_0}\epsilon_1([\gamma_i])\epsilon_2([\gamma_j^{-1}])\text{, if }i\le r_1<j\\
\frac1{1-\mu_0}\epsilon_2([\gamma_i])\epsilon_1([\gamma_j^{-1}])\text{, if }j\le r_1<i \end{cases}\]

To each word $g$ in the generators $\cl G$ and their inverses we define $\epsilon_E(g)$ as in Section \ref{SecAugReconstruct1}. Note that since $\epsilon_n(\mu) = \mu_0$ for both $n=1,2$ the definition of $E$ is such that the diagonal entries all agree. By Lemma \ref{LemAugFomE} then, this defines an augmentation $\epsilon:\cl C_{K_1\# K_2}\to\bb C$ provided that $\epsilon_E(\gamma_iR\gamma_j^{-1})=\epsilon_E(\gamma_i\gamma_j^{-1})$ for each $R\in\cl R \cup \cl R^{-1}$. We will need to use the following observation, which the reader may check using an argument similar to that in the proof of Lemma \ref{LemAugFomE}. 
\begin{enumerate}
\item[($\ast$)] If $\gamma_k^{-1}m\gamma_k=g_k\in\cl G_n$, and $h$ represents any element in $\pi_{K_{n'}}\subset\pi_{K_1\# K_2}$ where $n'\ne n$, then $\epsilon_E(\gamma_kh)=\frac1{1-\mu_0}\epsilon_n([\gamma_k])\epsilon_{n'}([h])$ and $\epsilon_E(h\gamma_k^{-1})=\frac1{1-\mu_0}\epsilon_n([\gamma_k^{-1}])\epsilon_{n'}([h])$.
\end{enumerate}

Let $R\in\cl R$ and consider $\epsilon_E(\gamma_iR\gamma_j^{-1})$ for some $1\le i,j\le r_1+r_2-1$. The check for $R\in\cl R^{-1}$ is essentially the same.

If both $i,j\le r_1$ and $R\in\cl R_1$ then $\epsilon_E(\gamma_iR\gamma_j^{-1})=\epsilon_1([\gamma_iR\gamma_j^{-1}])$ and this equals $\epsilon_1([\gamma_i\gamma_j^{-1}])=\epsilon_E(\gamma_i\gamma_j^{-1})$ since $\epsilon_1$ is well-defined. We have a similar argument when both $i,j>r_1$ and $R\in\cl R_2$.

If $i\le r_1<j$, then by ($\ast$) we see that $\epsilon_E(\gamma_iR\gamma_j^{-1})$ is either $\frac1{1-\mu_0}\epsilon_1([\gamma_iR])\epsilon_2([\gamma_j^{-1}])$ or $\frac1{1-\mu_0}\epsilon_1([\gamma_i])\epsilon_2([R\gamma_j^{-1}])$
depending on whether $R\in\cl R_1$ or $R\in\cl R_2$. In both cases this equals $\frac1{1-\mu_0}\epsilon_1([\gamma_i])\epsilon_2([\gamma_j^{-1}]) = \epsilon_E(\gamma_i\gamma_j^{-1})$. The case when $j\le r_1<i$ is similar. 

Finally, suppose both $i,j\le r_1$ but $R\in\cl R_2$. Here we use our assumption on the form of relators: that $R=g_{\ell}g_mg_{\ell}^{-1}g_k^{-1}$ for some $g_k, g_{\ell}, g_m\in\cl G_2$. Thus using our definition of $\epsilon$ and ($\ast$) we find that
{\small
\al{
\epsilon_E(\gamma_iR\gamma_j^{-1}) &=\epsilon_E(\gamma_ig_k^{-1}\gamma_j^{-1})-\epsilon_E(\gamma_ig_{\ell}\gamma_m^{-1})\epsilon_E(\gamma_mg_{\ell}^{-1}g_k^{-1}\gamma_j^{-1})\\
&=\epsilon_E(\gamma_i\gamma_j^{-1})+\mu_0^{-1}\epsilon_E(\gamma_i\gamma_k^{-1})\epsilon_E(\gamma_k\gamma_j^{-1})-\epsilon_E(\gamma_ig_{\ell}\gamma_m^{-1})\epsilon_E(\gamma_mg_{\ell}^{-1}g_k^{-1}\gamma_j^{-1})\\
&=\epsilon_E(\gamma_i\gamma_j^{-1})+\frac{\epsilon_1([\gamma_i])\epsilon_1([\gamma_j^{-1}])}{(1-\mu_0)^2}\left(\mu_0^{-1}\epsilon_2([\gamma_k^{-1}])\epsilon_2([\gamma_k])-\epsilon_2([g_{\ell}\gamma_m^{-1}])\epsilon_2([\gamma_mg_{\ell}^{-1}g_k^{-1}])\right)\\
&=\epsilon_E(\gamma_i\gamma_j^{-1})+\frac{\epsilon_1([\gamma_i])\epsilon_1([\gamma_j^{-1}])}{(1-\mu_0)^2}\left(\epsilon_2([g_k^{-1}])-\epsilon_2([e])-(\epsilon_2([g_k^{-1}])-\epsilon_2([g_{\ell}g_mg_{\ell}^{-1}g_k^{-1}]))\right)\\
&=\epsilon_E(\gamma_i\gamma_j^{-1}),
}
} 
the last equality since $\epsilon_2$ is well-defined. The case $i,j>r_1$ and $R\in\cl R_1$ is treated similarly. By Lemma \ref{LemAugFomE} we obtain a well-defined augmentation. 

Since the top-left $r_1\times r_1$ block of $E$ is the matrix $\cl E(\Gamma_1)$, we may choose $d_1$ from among the first $r_1$ columns of $E$ that are independent (the first being one of them). We may also choose $d_2-1$ columns from the last $r_2-1$ that (with the first column) also form an independent set. The union is an independent set by a standard argument. This shows that $K_1\# K_2$ has an augmentation of rank $d_1+d_2-1$.

Suppose that an augmentation $\epsilon:\cl C_{K_1\# K_2}\to\bb F$ with rank $d$ is given. Choose generators and elements of $\Gamma$ as above, ordered as above. Consider the KCH irrep $\rho:\pi_{K_1\# K_2}\to\text{GL}(W_{\rho}(\Gamma))$ corresponding to $\epsilon$ as constructed in Section \ref{SecKCHRepsandAugs}, where $W_{\rho}(\Gamma)$ is the vector space generated by the columns of $\cl E(\Gamma)$. By definition, $\rho(g)(\epsilon([\Gamma\gamma_j^{-1}])) = \epsilon([\Gamma g\gamma_j^{-1}])$ for any $g\in\pi_{K_1\# K_2}$. By Lemma \ref{LemEltExp} the restriction of $\rho$ to $\pi_{K_1}\subset\pi_{K_1\# K_2}$ is a \augrep on the subspace $W_1$ of $W_{\rho}(\Gamma)$ spanned by the first $r_1$ columns of $\cl E(\Gamma)$. Let $\epsilon_1$ be the induced augmentation. Similarly restriction of $\rho$ to $\pi_{K_2}$ gives a \augrep (inducing $\epsilon_2$ say), on the space $W_2$ spanned by the first column along with columns $r_1+1$ through $r_1+r_2-1$.

Consider the projection $\text{pr}(W_1)\subset\bb F^{r_1}$ of $W_1$ onto the first $r_1$ factors. By the definition of $E$, letting $E_j$ be column $j\in\{1,\ldots,r_1\}$ of $E$, any linear relation $\sum c_j\text{ pr}(E_j) = 0$ will also hold among the $E_j\in\bb F^{r_1+r_2-1}$. This makes it clear that $\text{pr}$ is an isomorphism between $W_1$ and the KCH irrep that induces $\epsilon_1$. In particular, $W_1$ is irreducible. An analogous statement holds for $W_2$. Thus $\dim W_n\le\text{ar}(K_n,\bb F)$ for $n=1,2$.

Since $W_{\rho}(\Gamma) = W_1+W_2$ and $W_1, W_2$ have a common 1-dimensional subspace, we see that $\dim W_{\rho}(\Gamma)\le\text{ar}(K_1,\bb F)+\text{ar}(K_2,\bb F)-1$. By Lemma \ref{DimWAndEpsDeg}, the rank of $\epsilon$ is at most $\text{ar}(K_1,\bb F)+\text{ar}(K_2,\bb F)-1$.
\end{proof}

\begin{cor}
If $\text{ar}(K_i,\bb F)=b(K_i)$ for $i=1,2$ (in which case $\text{mr}(K_i)=b(K_i)$), then $\text{ar}(K_1\# K_2,\bb F)=\text{mr}(K_1\# K_2)=b(K_1\# K_2) = b(K_1)+b(K_2)-1$.
\label{Brconnectsum}
\end{cor}

\begin{proof} Theorem \ref{ThmDimBound} shows that $\text{ar}(K_i,\bb F)\le\text{mr}(K_i)$ and it is well-known that $\text{mr}(K_i)\le b(K_i)$. Applying Theorem \ref{ThmConnSum} and observing that $b(K)-1$ is additive we obtain the result.
\end{proof}

Some knots that are known to satisfy the hypothesis of Corollary \ref{Brconnectsum} include torus knots, two-bridge knots, a family of pretzel knots \cite{C12}, and others by results in Section \ref{SecHRkAugs}.

\section{Augmentation rank of braid closures}
\label{SecHRkAugs}
In this section we construct augmentations by considering $K$ as the closure of a braid $B\in B_n$. To do so we pass to the algebra $HC_0(K)\vert_{U=1}$, discussed in Section \ref{SecBG} above, which is isomorphic as an $R_0$-algebra to $\cl C_K$. We begin showing how to understand augmentations with rank $n$, proving Theorem \ref{ThmWrithe}. This allows us to find knots with augmentation rank smaller than meridional rank. Afterwards we indicate a method to construct, from a knot that has a rank $n$ augmentation, a new knot with a rank $n+1$ augmentation. As a consequence we prove Theorem \ref{ThmFlypes}.

\subsection{Augmentations with rank equal to braid index}
\label{SecHRkAugs1}

Recall the isomorphisms $F_{\cl P}:\cl P_K\to\cl C_K$ and $F_{HC}:\cl C_K\to HC_0(K)\vert_{U=1}$ of Theorems \ref{ThmPi1Cords} and \ref{ThmHC0}, respectively. In the next lemma ${\bf A}$ is the $n\times n$ matrix used to define $HC_0(K)\vert_{U=1}$. Recall that ${\bf A}_{ij}=a_{ij}$ if $i< j$, ${\bf A}_{ij}=-\mu a_{ij}$ if $i>j$, and the diagonal entries are $1-\mu$.

\begin{lem} Let $\epsilon:\cl P_K\to\bb F$ be an augmentation. The rank of $\cl E(\Gamma)$ equals the rank of $\epsilon({\bf A})$.
\label{LemRkep(A)}
\end{lem}

We admit to abusing notation, as $\epsilon$ is defined on $\cl P_K$ not $HC_0(K)\vert_{U=1}$. By $\epsilon({\bf A})$ we actually mean $\epsilon\circ(F_{HC}\circ F_{\cl P})^{-1}({\bf A})$.

\begin{proof} 
Let $K$ be the closure of $B\in B_n$. Taking a basepoint $x$ for $\pi_K$, consider the generating set $\cl G=\set{g_1,g_2,\ldots,g_n}$ where $g_i$ is the meridian of $K$ contained in $D$, as depicted on the right in Figure \ref{FigAugTransfer} ($D$ is the disk from the discussion in Section \ref{SecHC}). For each $i>1$ define a loop $\gamma_i^{-1}\in\pi_K$ that follows $g_i$ in $D$ from $x$ until $g_i$ leaves the upper half-disk. Leaving $D$, $\gamma_i^{-1}$ then runs parallel to $K$, framed as in Figure \ref{FigAugTransfer}, until returning to $x$. By convention $\gamma_1$ is the identity.

\begin{figure}[ht]
\begin{tikzpicture}[scale=0.75,>=stealth]
	\draw (10,1) circle (50pt);
	\draw (10.6,2.25) node {$D$};
	\foreach \p in {-1,0,1,2} \filldraw (9.6+0.8*\p,1) circle (1pt);
	\draw (8.81,1-.02) node {$\ast$}
		  (9.1,1.07) node[below]{{\footnotesize $x$}};
	\draw (8.8,1) circle (6pt);
	\draw[->] (8.6,1.04) -- (8.6,1.06);
	\draw (8.8,0.9) node[below] {{\footnotesize $g_1$}};
	\draw (8.575+1.6,1) arc (180:360:6pt);
	\filldraw (9,1) circle (1pt);
	\draw	(9,1) ..controls (9,2) and (10.6,2) ..(10.6,1);
	\draw[-<]	(9,1) ..controls (9,1.75) and (10.175,1.75) ..(10.175,1);
	\draw 	(10.4,0.9) node[below] {{\footnotesize $g_i$}};
		
	\clip (-.7,-0.5) rectangle (6.6,2.75);
	\draw(0,-.25)--(0,2.25)--(3.5,2.25)--(3.5,-.25)--cycle;
	\draw(1.75,1) node {{\Large $B$}};
	\draw (5.25,1.99) node {{\large $\ast$}}
		  (4.7,1.95) node[below right] {{\footnotesize $x$}};
	\draw	(-.65,0.67) node {{\footnotesize $i$}}
			(5.1,0.63) node[above] {{\footnotesize $g_i$}};
	\draw[red] (4.6,0.65) node[below] {{\footnotesize $\gamma_i^{\textrm{-}1}$}};
	\draw	(3.5,2) -- (4.75,2);
	\foreach \y in {2}
		\draw[draw=white,double=black,line width=0.5pt]
		 	(4.75,\y+0.15) ..controls (4.6,\y+0.15) and (4.6,\y-0.15) ..(4.75,\y-0.15);
	\draw	(4.75,2+0.15) ..controls (4.9,2+0.15) and (4.9,2-0.15) ..(4.75,2-0.15);
	\draw[<-]
			(4.65,0.67+0.15) ..controls (5,0.67+0.15) and (4.9,0.67-0.15) ..(4.75,0.67-0.15);
	\foreach \y in {0,.67,1.33,2}
		\draw(-.35,\y) -- (0,\y);
	\foreach \y in {0,.67,1.33}	
		\draw (3.5,\y) -- (4.75,\y);
	\draw[->]	(4.75,2) -- (5.7,2);
	\foreach \y in {0,1.33}
		\draw (4.75,\y) -- (5.7,\y);
	\draw[draw=white,double=black] (4.75,0.67) -- (5.7,0.67);
	\draw	(5.7,2) ..controls (7.7+2*1.3,2) and (7.7+2*1.3,-2-1) ..(3,-2-1)
			 		  ..controls (-3-2*1.3,-2-1) and (-2.7-2*1.3,2)..(-.7,2);
	\foreach \y in {0,.67,1.33}
		\draw (5.7,\y) ..controls (7.7+\y*1.3,\y) and (7.7+\y*1.3,-\y-1) ..(3,-\y-1)
			 		  ..controls (-3-\y*1.3,-\y-1) and (-2.7-\y*1.3,\y)..(-.7,\y);
	\draw[draw=white,double=black,line width=0.5pt]
			(4.75,1.85) ..controls (4,1.85) and (4,0.52) ..(4.75,0.52);
	\draw	(5.25,1.99) node {{\large $\ast$}};
	\foreach \y in {0.52,1.18}
			\draw[red]	(5.7,\y) ..controls (7.7+\y*1.3,\y) and (7.7+\y*1.3,-\y-1) ..(3,-\y-1)
			 		  ..controls (-3-\y*1.3,-\y-1) and (-2.7-\y*1.3,\y)..(-.7,\y);	
	\draw[red,line width=0.65pt]
			(4.75,0.67-0.15) -- (5.7,0.67-0.15);
	\draw[red, line width=0.65pt]	(-.35,0.67-0.15) -- (0,0.67-.15)
				(3.5,1.33-.15)--(5.7,1.33-.15)
				(-.35,1.33-.15) -- (0,1.33-.15)
				(3.5,1.85) -- (4.6,1.85);
	\foreach \y in {0.52,1.18} \draw[red,line width=0.65pt,->]	(5.3,\y) -- (5.7,\y);
	\draw[draw=white,double=red]
			(4.75,1.85) ..controls (4,1.85) and (4,0.52) ..(4.75,0.52);
	\draw[draw=white,double=black,line width=0.5pt]	(4.75,1.85) ..controls (4.2,1.77) and (4.05,0.82) ..(4.65,0.82);
	\filldraw (4.75,1.85) circle (1.5pt);
\end{tikzpicture}
\caption{Computing $(F_{HC}\circ F_{\cl P})^{-1}(a_{ij})$}
\label{FigAugTransfer}
\end{figure}
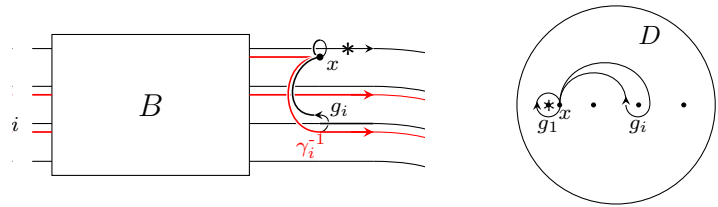

Define $\Gamma=\{\gamma_1,\ldots,\gamma_r\}$. By construction, $g_i$ and $\gamma_i^{-1}g_1\gamma_i$ are equal in $\pi_K$. Define $r_i$ to be the linking number $\text{lk}(\gamma_i,K)$. After homotopy of cords we see that $F_{\cl P}([\gamma_i\gamma_j^{-1}])=\mu^{r_i-r_j}c_{ij}$ and so $F_{HC}\circ F_{\cl P}([\gamma_i\gamma_j^{-1}])=\mu^{r_i-r_j}{\bf A}_{ij}$.

If $\Delta$ is the diagonal matrix with $\Delta_{ii}=\mu^{r_i}$ the discussion above implies $\cl E(\Gamma)\epsilon(\Delta)=\epsilon(\Delta)\epsilon({\bf A})$. As $\epsilon(\mu)\ne0$, $\epsilon({\bf A})$ and $\cl E(\Gamma)$ must have the same rank.
\end{proof}

\vspace*{6pt}
\noindent{\bf Theorem \ref{ThmWrithe}.}\ \ 
{\it Suppose that $K$ is the closure of $B\in B_n$, and that $\epsilon:\cl C_K\to\bb C$ is an augmentation of $K$ with rank $n$ and $\epsilon(\mu)=\mu_0$. Then $\epsilon(\lambda)=(-\mu_0)^{-w(B)}$, where $w(B)$ is the writhe (or algebraic length) of $B$. Furthermore, there is a curve of rank $n$ augmentations in the closure of $V_K$ that corresponds to a factor $\lambda\mu^{w(B)}-(-1)^{w(B)}$ of $\aug_K(\lambda,\mu)$.
}
\vspace*{6pt}

\begin{proof} The proof below uses an argument with determinants for which we work in the commutative algebra $\cl A_n^c$, defined as $\cl A_n$ modulo the ideal generated by $\setn{xy-yx}{x,y\in\cl A_n}$. Write $\cl I_B^c$ also for the ideal in $\cl A_n^c$ generated by the quotients of the elements in $\cl I_B$, the ideal of Definition \ref{DefHC_0}. For any augmentation $\epsilon:\cl C_K\to\bb C$, $\epsilon\circ F_{HC}^{-1}$ factors through $(\cl A_n^c\otimes R_0)/\cl I_B^c$. Without altering notation we suppose during the proof that, for example, the entries of a matrix which are traditionally in $\cl A_n$, are instead the corresponding class in $\cl A_n^c$.

Let $\Lambda$ be the diagonal matrix $\text{diag}[\lambda\mu^{w(B)},1,\ldots,1]$ used to define $HC_0(K)\vert_{U=1}$ (see Section \ref{SecHC}). That $\epsilon$ is well-defined implies
\[\epsilon({\bf A}) - \epsilon({\bf A})\epsilon(\Phi_B^R)\epsilon(\Lambda^{-1}) = 0.\]
By Lemma \ref{LemRkep(A)} the $n\times n$ matrix $\epsilon({\bf A})$ is invertible, and so $\epsilon(\Phi_B^R)=\epsilon(\Lambda)$. (Note also that $\epsilon(\Phi_B^L)=\epsilon(\Lambda)^{-1}$.)

By Corollary \ref{CorPhiInv}, $\Phi_B^R$ is invertible in the ring of matrices over $\ints[\{a_{ij}\}]$, its inverse being $\phi_B(\Phi_{B^{-1}}^R)$. So $\det\Phi_B^R$ is a unit in $\ints[\{a_{ij}\}]$ and can only be $\pm1$.

As $\epsilon(\Phi_B^R)=\epsilon(\Lambda)$ and $\det(\epsilon(\Phi_B^R))= \epsilon(\det(\Phi_B^R))=\pm1$, we see that $\epsilon(\lambda\mu^{w(B)})=\pm1$. To determine the sign, let $\sg_k$ be a standard generator of $B_n$. Use the definition of $\phi_{\sg_k}$ to check $\det(\Phi_{\sg_k^{\pm1}}^R)=-1$. Since $\phi_B$ is an algebra map we see that $\det(\phi_B(\Phi_{\sg_k^{\pm1}}^R))=-1$ for any $B\in B_n$. 

That $\det(\epsilon(\Phi_B^R))=(-1)^{w(B)}$ now follows from the Chain Rule (Theorem \ref{ThmChainRule}) and the fact that $w(B)$ has the same parity as the length of $B$ as a word in the generators $\{\sg_1,\ldots, \sg_{n-1}\}$ of $B_n$. It follows that $\epsilon(\lambda)=(-\mu_0)^{-w(B)}$.

Let $\Delta(B)$ be the $n\times n$ diagonal matrix $\text{diag}[(-1)^{w(B)},1,\ldots,1]$. We see that $\epsilon$ satisfies $\epsilon(\Phi_B^L)=\Delta(B)=\epsilon(\Phi_B^R)$. For any $\mu_1\in\bb C^*$, define a map $\epsilon_1$ on $\cl A_n\otimes R_0$ by setting $\epsilon_1(\mu)=\mu_1$, $\epsilon_1(\lambda)=(-\mu_1)^{-w(B)}$, and $\epsilon_1(a_{ij})=\epsilon(a_{ij})$ for $1\le i\ne j\le n$. Then $\epsilon_1(\Phi_B^L)=\Delta(B)=\epsilon_1(\Phi_B^R)$ and $\epsilon_1(\Lambda)=\Delta(B)$, and so $\epsilon_1$ defines an augmentation as
	\[\epsilon_1({\bf A})-\epsilon_1(\Lambda)\epsilon_1(\Phi_B^L)\epsilon_1({\bf A})=0=\epsilon_1({\bf A})- \epsilon_1({\bf A})\epsilon_1(\Phi_B^R)\epsilon_1({\bf A}).\]
Since $\det(\epsilon_1({\bf A}))$ is a polynomial in $\bb C[\mu_1]$, and non-zero when $\mu_1=\mu_0$, it can be zero for only finitely many $\mu_1$. Thus we have a rank $n$ augmentation for all but finitely many $\mu_1\in\bb C^*$ and the algebraic closure of $V_K$ contains a 1-dimensional component of rank $n$ augmentations which is the zero locus of $\lambda\mu^{w(B)}-(-1)^{w(B)}$.
\end{proof}

\begin{rem} The conventions adopted in \cite{Ng11HT} for infinity transverse homology would make $\epsilon(\lambda)=\mu_0^{w(B)}$, since the translation to our conventions is $\mu\leftrightarrow-\mu^{-1}$.
\label{RemConvs}
\end{rem}

\subsection{Rank gap}
\label{SecRankGap}

Let $n$ be the braid index of $K$ and write $\text{ar}(K)$ for $\text{ar}(K,\bb C)$. From (\ref{EqnIneqs}) the augmentation rank $\text{ar}(K)$ cannot be $n$ if $b(K)<n$, where $b(K)$ is the bridge number. Recall that the meridional rank $\text{mr}(K)$ is the minimal number of meridians that generate $\pi_K$. Ideas from the proof of Theorem \ref{ThmWrithe} allow us to find knots with a gap between their augmentation rank and meridional rank. That is, $\text{mr}(K)=b(K)=n$, but $\text{ar}(K)<n$. 

\begin{thm} Let $K$ be a knot with crossing number at most 10, which has bridge number and minimal braid index equal to three. If $K$ (or its mirror) is one of $8_{16},8_{17},10_{91},10_{94}$ then $2=\text{ar}(K) < \text{mr}(K) = 3$. Otherwise, $\text{ar}(K) = \text{mr}(K) = 3$.
\label{Thm3bridge3braids}
\end{thm}
\begin{proof}
For $B\in B_n$ recall the matrix $\Delta(B)$ defined in the proof of Theorem \ref{ThmWrithe}. If $K$ is the closure of $B$ and $\text{ar}(K)=n$, there is a map $\epsilon:\cl A_n\to\bb C$ with $\epsilon(\Phi_B^L)=\Delta(B)$. As the augmentation rank is an invariant of $K$, the choice of $B$ (provided it has index $n$) does not affect the existence of such a map. 

Define $\cl A_n^{ab}$ to be $\cl A_n^c$ modulo the ideal generated by $\setn{a_{ij}-a_{ji}}{1\le i\ne j\le n}$. Given $B\in B_n$, if $\epsilon:\cl A_n\to\bb C$ exists with $\epsilon(\Phi_B^L)=\Delta(B)$, and further $\epsilon$ descends to $\cl A_n^{ab}$, then $\epsilon$ defines an augmentation by Proposition \ref{Prop:Transpose} and Theorem \ref{ThmPhiSqRt}. Furthermore, for any extension of $\epsilon$ to $\cl A_n\otimes R_0$, a calculation of $\det(\epsilon({\bf A}))$ shows that its degree in $\epsilon(\mu)$ is $n>0$. So such an $\epsilon$ determines a family of rank $n$ augmentations for all but finitely many choices for $\epsilon(\mu)$.

Consulting the database at KnotInfo \cite{knot-info}, there are 42 prime knots (up to mirroring) which have bridge number and braid index three and have crossing number at most 10. For each we take a braid representative $B\in B_3$ (e.g.\ take the representative provided at \cite{knot-info}). Let $J_B$ be the ideal in $\cl A_3^c$ generated by the polynomials appearing as entries in $\Phi_B^L - \Delta(B)$ (there is no issue working in $\cl A_3^c$ rather than $\cl A_3$ as $\epsilon$ must factor through $\cl A_3^c$ in any case). We calculate $\Phi_B^L - \Delta(B)$ and compute a reduced Gr\"obner basis for $J_B$.\footnote{The reader is referred to the Mathematica notebook 3bridge3braids.nb, available at \href{http://tigerweb.towson.edu/ccornwell/docs/aux_files.html}{the author's website}. Our computation of $\Phi_B^L - \Delta(B)$ uses the Mathematica package transverse.m written by Lenny Ng and found at \href{http://www.math.duke.edu/~ng/math/programs.html}{his website}. Then we use the Mathematica function GroebnerBasis. Note that the transverse.m was written for transverse homology, so one would consider entries in $\Phi_B^L-\text{Id}$ by Remark \ref{RemConvs}.} 

For the braid representatives of $8_{16},8_{17},10_{91},10_{94}$, the Gr\"obner basis of $J_B$ is $\{1\}$, so the polynomials have no common zero in $\bb C$ and there does not exist a map $\epsilon:\cl A_n\to\bb C$ with $\epsilon(\Phi_B^L)=\Delta(B)$. For each of the other braid representatives one can see that the computed basis of polynomials has a common zero on which $a_{ij}-a_{ji}$, $1\le i<j\le 3$, evaluates to zero. By the above discussion we get a rank 3 augmentation.

It was shown in \cite{DGa04} and \cite{BZ05} that every non-trivial knot has a non-abelian SL$_2\bb C$ representation, and so $\text{ar}(8_{16})=\text{ar}(8_{17})=\text{ar}(10_{91})=\text{ar}(10_{94})=2$. As shown in \cite{BZ89}, a consequence of Thurston's orbifold geometrization and the fact that only 2-bridge knots have double branched cover a lens space (\cite{HR85}) is that 3-bridge knots have meridional rank 3.

For the mirror $m(K)$ of any knot $K$ that has already been checked, \cite[Prop.\ 4.2]{Ng08} implies that the cord algebra of $m(K)$ is isomorphic to the algebra obtained from $HC_0(K)\vert_{U=1}$ by transforming $\mu\leftrightarrow\mu^{-1}$. Thus $\text{ar}(m(K)) = 3$ if and only if $\text{ar}(K) = 3$. Finally, if $K$ is not prime then since $b(K)=3$ and $b(K)-1$ is additive under connect sum, $K$ is the connect sum of two 2-bridge knots, each of which has augmentation rank 2. Then $\text{ar}(K) = 3$ by Theorem \ref{ThmConnSum}.
\end{proof}

\subsection{Increasing augmentation rank with the braid index}
\label{SecHRkAugs2}

Let $K$ be a knot with braid index $n$ and a rank $n$ augmentation $\cl A_n\to\bb C$ that descends to $\cl A_n^{ab}$. We describe a method for constructing knots with higher braid index that have a rank $n+1$ augmentation. Let $\sg_{i,j}$, for $i<j$, denote the braid $\sg_i\ldots\sg_{j-2}\sg_{j-1}\sg_{j-2}^{-1}\ldots\sg_i^{-1}$ (visually, this braid crosses the $i$ and $j$ strands above the intermediate strands). We will prove the following.

\begin{thm} Let $u,v$ be integers with $\abs u\ge 2$, $\abs v\ge 3$, and let $\delta=\pm1$. If the closure of $B\in B_n$ has a rank $n$ augmentation that descends to $\cl A_n^{ab}$ then, for $1\le i< n$, the closure of $B\sg_n^{\delta}\sg_{i,n}^u\sg_n^v\in B_{n+1}$ has a rank $n+1$ augmentation with the same property, provided $u+v-1$ is even or $i=1$ and $u$ is odd.
\label{ThmBuildingRknAugs}
\end{thm}

\begin{figure}[ht]
\begin{tikzpicture}[scale=0.65,>=stealth]
	\foreach \c in {0,5,10} \draw (\c,1) circle (50pt);
	\foreach \c in {0,5,10}
		\foreach \p in {-1,1} \filldraw (\c+\p,1) circle (1.5pt);
	\foreach \c in {0,5,10}	
		\draw 	(\c+1.75,1) node {{\large $\ast$}}
				(\c+0.6,2.25) node {$D$}
				(\c-1,1) node[below] {{\small $i$}}
				(\c+1,1) node[below] {{\small $j$}};
	\foreach \c in {0,10}
		\draw	(\c,1) node {$\ldots$};
	\filldraw	(5,1) circle (1.5pt);
	\draw	(5,1) node[below] {{\small $k$}};
	\draw[->,thick] (9,1.1) ..controls (9,1.4) and (10.7,2.2) .. (11.685,1.01);
	\draw (10,-0.75) node[below] {{\small $\sg_{i,j}\cdot c_{j,n+1}$}};
	\draw[->,thick] (1,1.125) ..controls (0.75,1.625) and (-0.25,1.4) ..(-0.75,0.9)
							  ..controls (-1.1,0.69) and (-1.45,0.99) ..(-1.1,1.2)
							  ..controls (-0.75,1.41) and (0.7,2.2) ..(1.685,1.01);
	\draw (0,-0.75) node[below] {{\small $\sg_{i,j}\cdot c_{i,n+1}$}};	
	\draw[->,thick] (5,1.1) ..controls (5.2,1.3) and (5.25,1.4) ..(5.75,0.9)
							..controls (6.1,0.69) and (6.45,0.99) ..(6.1,1.2)
							..controls (5.2,1.75) and (4.95,1.32)	..(5-0.75,0.9)
							..controls (5-1.1,0.69) and (5-1.45,0.99) ..(5-1.1,1.2)
							..controls (5-0.4,1.62) and (5.7,2.2) ..(6.685,1.01);
	\draw (5,-0.75) node[below] {{\small $\sg_{i,j}\cdot c_{k,n+1}$}};
\end{tikzpicture}
\caption{Computing $\Phi_{B'}^L$}
\label{FigPhiLCalcA}
\end{figure}
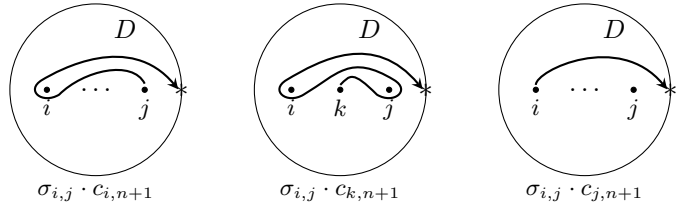

For the proof of Theorem \ref{ThmBuildingRknAugs} we first require a pair of lemmas.

\begin{lem} Let $B'=\sg_{i,j}^u\in B_n$ for some $1\le i< j\le n$ and $\abs u\ge 2$ and let $X=\{a_{ij},a_{ji}\}$. If $u$ is odd or $u+v-1$ is even then there exists a map $e:\cl A^{ab}_n\to\cl A^{ab}_n\otimes\bb C$ such that $e(a_{kl})=a_{kl}$ for $a_{kl}\not\in X$, and such that $e(\Phi_{B'}^L)$ is the $n\times n$ diagonal matrix $\text{diag}[\text{Id}_{i-1},(-1)^{u+v-1},\text{Id}_{j-i-1},(-1)^{v-1},\text{Id}_{n-j}]$.
\label{LemBasePhis}
\end{lem}
\begin{proof}
We give a proof for the case $u\ge 2$. Writing a verbatim copy of this proof, but interchanging the roles of $i$ and $j$ throughout (with the one exception that $\sg_{i,j}$ is simply replaced with $\sg_{i,j}^{-1}$) one obtains a proof for the case $u\le-2$. 

Include $B_n\subset B_{n+1}$ and consider $\phi_B\in\text{Aut}(\cl A_{n+1})$. Placing the $n+1$ puncture on $D$ at the boundary we note that the class of $\sg_{i,j}\cdot c_{i,n+1}$ is represented by the leftmost arc in Figure \ref{FigPhiLCalcA}. Using relations from Figure \ref{FigRelnPathAlg}, we compute $\phi_{\sg_{i,j}}(a_{i,n+1}) = a_{j,n+1}-a_{ji}a_{i,n+1}$. The arc on the right of Figure \ref{FigPhiLCalcA} is $\sg_{i,j}\cdot c_{j,n+1}$ and so $\phi_{\sg_{i,j}}(a_{j,n+1})=a_{i,n+1}$. We also find $\phi_{\sg_{i,j}}(a_{ij})=-a_{ji}$ and $\phi_{\sg_{i,j}}(a_{ji})=-a_{ij}$. 

First, $\phi_{\sg_{i,j}}$ descends to $\cl A_n$ modulo the ideal generated by $a_{ij}-a_{ji}$. Write $x=a_{ij}$, then $\phi_{\sg_{i,j}}(x)=-x$. In addition, for any $k\ge 1$, we see $\phi_{\sg_{i,j}^k}(a_{i,n+1})=P_k(x)a_{i,n+1} + Q_k(x)a_{j,n+1}$ for some polynomials $P_k,Q_k$ in $\ints[x]$. Using a computation above, we have $P_0(x) = 1$, $Q_0(x) = 0$, $P_1(x) = -x$, and $Q_1(x) = 1$. Now we compute
	\al{
	\phi_{\sg_{i,j}^{k+1}}(a_{i,n+1}) &= \phi_{\sg_{i,j}}(P_k(x)a_{i,n+1}+Q_k(x)a_{j,n+1})\\
		&= P_k(-x)(a_{j,n+1}-xa_{i,n+1})+Q_k(-x)a_{i,n+1},
	}
so $Q_{k+1}(x) = P_k(-x)$, and also
	\begin{equation}
	P_{k+1}(x)=P_{k-1}(x)-xP_k(-x).
	\label{EqnRecurrP_k}
	\end{equation}
The recurrence in (\ref{EqnRecurrP_k}) with initial data $P_0(x)=1$ and $P_1(x)=-x$ then determines $P_k(x)$ (and $Q_k(x) = P_{k-1}(-x)$) for all $k\ge1$. Note that $P_k$ is an odd or even function, when $k$ is odd or even respectively, that the degree of $P_k$ is $k$, and that $\phi_{\sg_{i,j}^k}(a_{j,n+1}) = \phi_{\sg_{i,j}^{k-1}}(a_{i,n+1}) = P_{k-1}(x)a_{i,n+1}+P_{k-2}(-x)a_{j,n+1}$.

By our computations rows $i$ and $j$ of $\Phi_{B'}^L$ consist of only four non-zero entries (in columns $i$ and $j$), namely: $P_u(x), P_{u-1}(-x)$ in row $i$; $P_{u-1}(x)$, $P_{u-2}(-x)$ in row $j$. For $k<i$ or $k>j$ we have that $\phi_{B'}(a_{k,n+1})=a_{k,n+1}$. To understand the intermediate rows, let $i<k<j$ and check that the center of Figure \ref{FigPhiLCalcA} gives $\sg_{i,j}\cdot c_{k,n+1}$. Then
\al{
\phi_{\sg_{i,j}}(a_{k,n+1}) 
							&= (a_{k,n+1}-a_{ki}a_{i,n+1})-a_{kj}\phi_{\sg_{i,j}}(a_{i,n+1})\\
							&= (a_{k,n+1}-a_{ki}a_{i,n+1})+\phi_{\sg_{i,j}}(a_{ki}a_{i,n+1}),
}
since $\phi_{\sg_{i,j}}(a_{ki})=-a_{kj}$. But then $\phi_{\sg_{i,j}}(a_{k,n+1}-a_{ki}a_{i,n+1}) = a_{k,n+1}-a_{ki}a_{i,n+1}$, implying $\phi_{\sg_{i,j}^u}(a_{k,n+1}) = (a_{k,n+1}-a_{ki}a_{i,n+1}) + \phi_{\sg_{i,j}^u}(a_{ki}a_{i,n+1})$. And so the $k^{th}$ row of $\Phi_{B'}^L$ consists of 
	\[\phi_{B'}(a_{ki})P_u(x)-a_{ki},\ 1,\ \text{and}\ \phi_{B'}(a_{ki})P_{u-1}(-x)\] in the $i,k,j$ spots, respectively (and 0 elsewhere).

To prove the lemma then, it suffices to show there is a choice for $e(x)=x_0\in\bb C$ that makes $P_u(x_0)=(-1)^{u+v-1}, P_{u-1}(x_0)=0, P_{u-2}(-x_0)=(-1)^{v-1}$ and $\phi_{B'}(a_{ki})\vert_{x=x_0}=(-1)^{u+v-1}a_{ki}$ for $i<k<j$ (recall that $P_{u-1}$ is either an odd or even function of $x$). 

As $u-1>0$ there is $x_0$ for which $P_{u-1}(x_0)=0$. Choose $e(x)=x_0$ and $e(a_{kl})=a_{kl}$ for $a_{kl}\not\in X$. Then only the $i^{th}$ column of $e(\Phi_{B'}^L)$ can possibly have non-zero off-diagonal entries. Hence $\det e(\Phi_{B'}^L)=P_u(x_0)P_{u-2}(-x_0)$, as other diagonal entries are 1. In the proof of Theorem \ref{ThmWrithe} we saw that this determinant is $(-1)^{w(B')}=(-1)^u$. Moreover, from equation (\ref{EqnRecurrP_k}), $P_u(x_0) = P_{u-2}(x_0) = (-1)^{u-2}P_{u-2}(-x_0)$. Putting this all together requires that $P_{u-2}(-x_0)=\pm1$.

If $u$ is odd then $P_{u-2}$ is an odd function, so we can guarantee $P_{u-2}(-x_0)=(-1)^{v-1}$ by swapping $x_0$ for $-x_0$ if needed. As $P_{u-1}$ is even it remains zero, and we have $P_u(x_0)=(-1)^{u-2}P_{u-2}(-x_0)=(-1)^{u+v-1}$. If $u+v-1$ is even we need only consider $u$ even. In this case take $x_0=0$ and this makes $P_u(x_0)=1=P_{u-2}(-x_0)$ as needed.

Having determined the diagonal, consider $e(\phi_{B'}(a_{ki}))$. Since $k>i$, Theorem \ref{ThmPhiSqRt} implies that $-\mu e(\phi_{B'}(a_{ki}))$ equals the $(k,i)$-entry of $e(\Phi_{B'}^L\cdot{\bf A}\cdot\Phi_{B'}^R)$. By calculations above, row $k$ of $e(\Phi_{B'}^L)$ has at most two non-zero entries and column $i$ of $e(\Phi_{B'}^R)$ has one, since row $i$ of $e(\Phi_{B'}^L)$ does. So the $(k,i)$-entry in question is $(-1)^{u+v-1}((e(\phi_{B'}(a_{ki}))P_u(x_0)-a_{ki})(1-\mu)-\mu a_{ki})$. Equating this with $-\mu e(\phi_{B'}(a_{ki}))$ and using that $P_u(x_0)=(-1)^{u+v-1}$ we find that, as required, $e(\phi_{B'}(a_{ki})) = (-1)^{u+v-1}a_{ki}$.
\end{proof}

\begin{lem}
Consider $B,B'\in B_n$ included into $B_{n+1}$ so the last strand does not interact and with $\Phi_B^L$ and $\Phi_{B'}^L$ as $(n+1)\times(n+1)$ matrices. Let $b=B\sg_n^{-1}B'\sg_n^v$ for some $v\ge 3$. Suppose there exists a map $e:\cl A_{n+1}^{ab}\to\bb C[X]$ where $e(a_{n,n+1})=e(a_{n+1,n})=X$, with the properties
\begin{itemize}
	\item[(a)]$e(\Phi_B^L)=\Delta(B)$;
	\item[(b)]$e(\phi_{B\sg_n^{-1}}(\Phi_{B'}^L)) = \text{diag}[(-1)^{w(B')+v-1},1\ldots,1,(-1)^{v-1},1]$.
\end{itemize}
Then there is a map $\bar e:\bb C[X]\to\bb C$ such that $\bar e\circ e(\Phi_{b}^L)=\Delta(b)$.
\label{LemFlype}
\end{lem}

\begin{proof}
We treat the case when $v$ is positive first. The $n^{th}$ row of $e(\Phi_B^L)$ being $(0,\ldots,0,1,0)$ implies $e(\phi_B(a_{n,n+1})) = X$. By the Chain Rule we have
{\small
\[e\left(\Phi_{B\sg_n^{-1}}^L\right) = \begin{pmatrix}\text{Id}_{n-1}&&\\ &0&1\\ &1&-X\end{pmatrix}\Delta(B) = \begin{pmatrix}(-1)^{w(B)}&&&\\ &\text{Id}_{n-2} && \\ && 0&1\\ &&1 & -X\end{pmatrix}.\]
}

By property (b), the Chain Rule, and $w(b)=w(B)+w(B')+v-1$ we get,
{\small
\[e\left(\Phi_{B\sg_n^{-1}B'}^L\right) = e\left(\phi_{B\sg_n^{-1}}(\Phi_{B'}^L)\cdot\Phi_{B\sg_n^{-1}}^L\right) =\begin{pmatrix}(-1)^{w(b)}&&&\\ &\text{Id}_{n-2} && \\ && 0&(-1)^{v-1}\\ &&1&-X\end{pmatrix}.\]
}

If we show that $\bar e(X)$ may be chosen so that 
{\small
\[\bar e\circ e\left(\phi_{B\sg_n^{-1}B'}(\Phi_{\sg_n^v}^L)\right) = \begin{pmatrix}\text{Id}_{n-1} &&\\ &(-1)^{v-1}\bar e(X) & 1\\ & (-1)^{v-1} & 0\end{pmatrix},\]
}

then applying the Chain Rule again will prove the claim. First note that $\sg_{n,n+1}=\sg_n$, so upon setting $a_{n+1,n}=a_{n,n+1}=x$ the proof of Lemma \ref{LemBasePhis} implies that 
{\small
\begin{equation}\Phi_{\sg_n^v}^L= \begin{pmatrix}\text{Id}_{n-1} &&\\ &P_v(x) & P_{v-1}(-x)\\ & P_{v-1}(x) & P_{v-2}(-x)\end{pmatrix}.
\label{EqnTwistPhiMatrix}
\end{equation}
}

To understand $e(\phi_{B\sg_n^{-1}B'}(a_{n,n+1}))$ we use Theorem \ref{ThmPhiSqRt} to make a calculation similar to the one at the end of the proof of Lemma \ref{LemBasePhis}. Set $b'=B\sg_n^{-1}B'$ and consider the equation $\phi_{b'}({\bf A}) = \Phi_{b'}^L{\bf A}\Phi_{b'}^R$. With our knowledge of $e(\Phi_{b'}^L)$ and Proposition \ref{Prop:Transpose} we find that
			\[e(\phi_{b'}(a_{n,n+1})) = (-1)^{v-1}(-\mu e(a_{n+1,n})-(1-\mu)e(a_{n,n+1}))=(-1)^{v}X.\]

Recalling that $P_v$ is odd or even as $v$ is odd or even, this implies that $e\left(\phi_{b'}(\Phi_{\sg_n^v}^L)\right)$ is the matrix in (\ref{EqnTwistPhiMatrix}) after substituting $P_{v-1}(\pm x)\mapsto P_{v-1}(\pm X)$, $P_v(x)\mapsto(-1)^vP_v(X)$, and $P_{v-2}(-x)\mapsto(-1)^vP_{v-2}(-X)$. 

As $v-1>0$, there is a choice $\bar e(X)=X_0\in\bb C$ so that $P_{v-1}(X_0)= (-1)^{v-1}$. Now $\Phi_{\sg_n^{v-1}}^L=\Delta(\sg_n^{v-1})$ has a solution when $v\ge3$, so $X_0$ may be chosen so that $P_{v-2}(X_0)=0$. Then $(-1)^vP_v(X_0)=(-1)^{v-1}X_0(-1)^{v-1}P_{v-1}(X_0)=(-1)^{v-1}X_0$ by (\ref{EqnRecurrP_k}).

Finishing the proof, if we extend $\bar e$ to $\bb C[X]$ then
{\small
	\al{
	\bar e\circ e(\Phi_b^L)	&= \bar e\circ e\left(\phi_{B\sg_n^{-1}B'}(\Phi_{\sg_n^v}^L)\right)\cdot\bar e\circ e\left(\Phi_{B\sg_n^{-1}B'}^L\right)\\
						&= \begin{pmatrix}\text{Id}_{n-1} &&\\ &(-1)^{v-1}\bar e(X) & 1\\ & (-1)^{v-1} & 0\end{pmatrix}\begin{pmatrix}(-1)^{w(b)}&&&\\ &\text{Id}_{n-2} && \\ && 0&(-1)^{v-1}\\ &&1&-\bar e(X)\end{pmatrix} = \Delta(b).
	}
}

The only difference when $v\le -3$ is that the switching of $i$ and $j$ in the proof of Lemma \ref{LemBasePhis} means the matrix in (\ref{EqnTwistPhiMatrix}) should instead be
{\small
\begin{equation}\Phi_{\sg_n^v}^L= \begin{pmatrix}\text{Id}_{n-1} &&\\ &P_{-v+2}(-x) & P_{-v+1}(x)\\ & P_{-v+1}(-x) & P_{-v}(x)\end{pmatrix}.
\label{EqnTwistPhiMatrixneg}
\end{equation}
}

As $v\le -3$ it is possible to choose $X_0$ so that $P_{-v+1}(-X_0)=(-1)^{v-1}$ and $P_{-v}(X_0)=0$ as in the positive case.
\end{proof}

\begin{proof}[Proof of Theorem \ref{ThmBuildingRknAugs}]

Lemmas \ref{LemBasePhis} and \ref{LemFlype} handle much of the work. Suppose that $\epsilon:\cl A_n^{ab}\to\bb C$ determines a rank $n$ augmentation on $HC_0(K)\vert_{U=1}$, and $K$ is the closure of $B\in B_n$. Define $B'=\sg_{i,n}^u$. We compare $\epsilon(\phi_{B\sg_n^{-1}}(\Phi_{B'}^L))$ to $\epsilon(\Phi_{B'}^L)$ with an argument like that used in Lemma \ref{LemFlype}.

By calculations made in the proof of Lemma \ref{LemBasePhis}, along with the fact that $\phi_{\sg_{i,n}}(a_{kn}) = a_{ki} - a_{kn}a_{ni}$ (and the analgous identity for $\phi_{\sg_{i,n}}(a_{nk})$), the matrix $\Phi_{B'}^L$ has entries involving only $a_{ik}$, $a_{ki}$, $a_{kn},a_{nk}$ and $a_{in},a_{ni}$, for $i<k<n$. 

We note that for any $j<n$, \al{\phi_{\sg_n^{-1}}(a_{jn}) &= a_{j,n+1},\ \phi_{\sg_n^{-1}}(a_{nj}) = a_{n+1,j}\\ \text{ and }\quad\phi_{\sg_n^{-1}}(a_{ki}) &= a_{ki},\ \phi_{\sg_n^{-1}}(a_{ik})=a_{ik}.} 
Extend $\epsilon$ to $\cl A_{n+1}^{ab}$ so that $\epsilon(a_{j,n+1})=a_{j,n+1}$ for $1\le j\le n$. The matrices $\Phi_B^L$ and $\Phi_B^R$, by definition, record the image under $\phi_B$ of $a_{j,n+1}$ and $a_{n+1,j}$ respectively. We observed in Theorem \ref{ThmWrithe} that $\epsilon(\Phi_B^L)=\Delta(B)=\epsilon(\Phi_B^R)$. Moreover, we note that $\epsilon(\phi_B(a_{ik}))$ is either $a_{ik}$ or $(-1)^{w(B)}a_{ik}$ according to whether $i>1$ or $i=1$, since $\epsilon(\phi_B({\bf A})) = \epsilon({\bf\Lambda^{-1}}\cdot{\bf A}\cdot{\bf \Lambda})$. A similar statement holds for $a_{ki}$.

If $i>1$ and $u+k-1$ is even, then the previous paragraph and Theorem \ref{ThmPhiSqRt} imply that $\epsilon(\phi_{B\sg_n^{-1}}(\Phi_{B'}^L))$ is the matrix $\Phi_{B'}^L$, but with $a_{jn}$ (resp. $a_{nj}$) replaced by $a_{j,n+1}$ (resp. $a_{n+1,j}$) for $j=i,k$. Now Lemma \ref{LemBasePhis} gives a map $e:\cl A_{n}^{ab}\to\cl A_{n}^{ab}\otimes\bb C$ such that {\small \[e(\Phi_{B'}^L)=\text{diag}[\text{Id}_{i-1},(-1)^{u+v-1},\text{Id}_{n-i-1},(-1)^{v-1}]=\text{diag}[(-1)^{u+v-1},\text{Id}_{n-2},(-1)^{v-1}]\]} as $u+v-1$ is even. 

Define $\bar e$ on $\cl A_{n+1}^{ab}$ so that $\bar e(a_{j,n+1})=e(a_{j,n})$ for $j=i,k$, and $\bar e(a_{n,n+1})=X$. Now define $\bar\epsilon=\bar e\circ \epsilon:\cl A_{n+1}^{ab}\to\bb C[X]$. Then by Lemma \ref{LemFlype}, since $\bar\epsilon$ satisfies properties (a) and (b), it determines a rank $n+1$ augmentation for the closure of $B\sg_n^{-1}B'\sg_n^v$.

If $i=1$ and $u$ is odd then $\epsilon(\phi_{B\sg_n^{-1}}(\Phi_{B'}^L))$ is $\Phi_{B'}^L$ but with $a_{1n}$ (resp. $a_{n1}$) replaced by $(-1)^{w(B)}a_{1,n+1}$ (resp. $(-1)^{w(B)}a_{n+1,1}$). The fact that the polynomials $P_u$, $P_{u-2}$ on the diagonal of $\Phi_{B'}^L$ are odd functions, along with the relation in (\ref{EqnRecurrP_k}), allows us to choose the sign of $P_u(\bar e(x))$, with $P_{u-2}(-\bar e(x))$ having opposite sign (in similar fashion to the proof of Lemma \ref{LemBasePhis}). This means that we may find $\bar\epsilon$, in similar fashion to the case $i>1$, satisfying the hypotheses of Lemma \ref{LemFlype}.

This proves the theorem for the case that $\delta=-1$. When $\delta=1$ the closure of $B\sg_n^\delta B'\sg_n^v$ is the mirror of a knot $K$ for which the theorem is already proved. By \cite[Prop.\ 4.2]{Ng08}, the mirror of $K$ has a rank $n+1$ augmentation if $K$ has a rank $n+1$ augmentation. This finishes the proof.
\end{proof}

\vspace*{6pt}
\noindent{\bf Corollary \ref{ThmFlypes}.}\ \ 
{\it If $\abs u, \abs w\ge 2$, $\abs v\ge 3$, and $\delta=\pm1$ and a knot $K$ is the closure of $b=\sg_1^w\sg_2^{\delta}\sg_1^u\sg_2^v$, then the closure of $V_K$ contains a curve of rank 3 augmentations.}
\vspace*{6pt}

\begin{proof} If $u$ is odd then this follows from Theorem \ref{ThmBuildingRknAugs} by taking $n=2$ and $i=1$. If $u$ is even then $w$ must be odd as $K$ is a knot. There is either a positive or negative flype (according to the sign of $\delta$) taking $b$ to the braid $\sg_1^u\sg_2^{\delta}\sg_1^w\sg_2^v$, which also has closure $K$. Now apply Theorem \ref{ThmBuildingRknAugs} to this braid.
\end{proof}

\address{Department of Mathematics, Towson University\\ 
8000 York Road, Towson, MD 21252, USA\\
\email{ccornwell@towson.edu}\\

\end{document}